\documentclass[11pt]{amsart}
\usepackage{amsmath,amsfonts,graphicx}
\usepackage{amssymb}
\usepackage{comment}
\usepackage{color}
\usepackage{tikz}
\usepackage{float}
\usepackage[caption = false]{subfig}
\newtheorem{thm}{Theorem}[section]
\newtheorem{prp}[thm]{Proposition}
\newtheorem{lmm}[thm]{Lemma}
\newtheorem{cor}[thm]{Corollary}
\newtheorem{dfn}[thm]{Definition}
\newtheorem{rmk}[thm]{Remark}
\newtheorem{cjt}[thm]{Conjecture}

\newcommand{\thomas}[1]{{\color{red}*}\marginpar{\tiny  \color{red} TR: #1}}

\renewcommand{\theenumi}{\arabic{enumi}}

\title{Pseudo-monodromy and the Mandelbrot set}
\author{Yutaka Ishii and Thomas Richards}
\address{Y.I. \& T.R., Department of Mathematics, Kyushu University, Motooka, Fukuoka 819-0395, Japan. Emails: \texttt{yutaka@math.kyushu-u.ac.jp}, \texttt{richards.jordan.thomas.745@m.kyushu-u.ac.jp}}

\begin{document}

\begin{abstract}
We investigate the discontinuity of codings for the Julia set of a quadratic map. To each parameter ray,  we associate a natural coding for Julia sets on the ray. Given a hyperbolic component $H$ of the Mandelbrot set, we consider the codings along the two parameter rays landing on the root point of $H$. Our main result describes the discontinuity of these two codings in terms of the kneading sequences of the hyperbolic components which are conspicuous to $H$. This result can be interpreted as a solution to the degenerate case of the monodromy action conjecture in the horseshoe locus for the complex H\'enon family.

\bigskip

\end{abstract} 

\maketitle


\section{Introduction}

For polynomial maps of degree $d\geq 2$ in one complex variable, the monodromy problem aims to understand how Julia sets vary as one traverses loops in the degree $d$ shift locus. Since the work of Blanchard, Devaney and Keen~\cite{BDK}, this problem is well-understood; any loop in the shift locus induces an automorphism of the one-sided $d$-shift, and conversely any automorphism of the one-sided $d$-shift is realised by a loop in the shift locus. This gives insight into the topological structure of the degree $d$ shift locus.

In the 1990's, Hubbard conjectured an analogous correspondence for the complex H\'enon family, between the monodromy actions along loops in the hyperbolic horseshoe locus and  automorphisms of the bi-infinite shift space on two symbols. Based on some numerical experiments, Lipa~\cite{Li} formulated a more detailed conjecture, called the \emph{monodromy action conjecture}, which relates the monodromy action along an individual loop in the hyperbolic horseshoe locus to certain bifurcation structures in the parameter space of the complex H\'enon family (see Appendix \ref{sec:Henon} for more details). This conjecture predicts the existence of a surprising bridge between the dynamical space and the parameter space for the complex H\'enon dynamics (compare the so-called ``magic formula'' of Douady and Hubbard for the quadratic family~\cite{DH1, DH2}). 

However, a proof of the monodromy action conjecture remains beyond our reach. This is partly because, unlike the one-dimensional case, no generating set is known for the bi-infinite $2$-shift, and the technique of surgery used in~\cite{BDK} is not applicable to complex two-dimensional dynamics. Moreover, some of the definitions employed by Lipa to formulate the conjecture, such as the notion of \emph{herds} introduced by Koch  (see~\cite{Li}), 
lack necessary mathematical rigour.

The purpose of the present paper is to provide a solution to the complex one-dimensional analogue of the monodromy action conjecture in an appropriate framework. Recall that the  Mandelbrot set $\mathfrak{M}$ is known to be connected~\cite{DH1, DH2}, therefore the only non-trivial loop in the shift locus $\mathbb{C}\setminus\mathfrak{M}$ is the one surrounding the Mandelbrot set $\mathfrak{M}$.
To formulate the monodromy action problem \textit{\`a la} Lipa under such circumstances, we consider two parameter rays in $\mathbb{C}\setminus\mathfrak{M}$ landing on the same root point $r_H$ of a hyperbolic component $H$ of $\mathfrak{M}$. Then, the discontinuity between the codings for Julia sets as one approaches $r_H$ along the two parameter rays can be interpreted as the ``pseudo-monodromy action'' along the ``pseudo-loop'' passing through the root point $r_H$ (see Section \ref{sec:monodromy} for more details). 

One of our main results (Corollary \ref{cor:main}) states that this discontinuity can be completely characterised in terms of the kneading sequences of the hyperbolic components \emph{conspicuous} to $H$ (see Section \ref{sec:kneading} for definition). Moreover, in Appendix \ref{sec:Henon} we explain how this result can be regarded as a solution to the degenerate case of the monodromy action conjecture, and the relevance to the full monodromy action conjecture. 
We remark that notions such as herds, which were not established rigorously in the formulation of the monodromy action conjecture do not appear in the degenerate case and as such our problem is well-posed.
The degenerate case was also considered by Lipa (see Section 5.6 in~\cite{Li}), but only for points with itineraries not containing $\star$, simplifying the analysis considerably.
We also note that Corollary \ref{cor:main} generalises the previous work by Atela~\cite{A1, A2} on the bifurcation of parameter rays. 

The organisation of this paper is as follows. In the next section, the monodromy problem for the quadratic family is formulated in terms of pseudo-loops in $\mathbb{C}\setminus\mathfrak{M}$. In Section \ref{sec:kneading} we recall some combinatorial aspects of the Mandelbrot set $\mathfrak{M}$ and introduce the notion of conspicuousness. Section \ref{sec:main} is dedicated to stating the main results, and their proofs are given in Section \ref{sec:proof}. In Appendix \ref{sec:examples} we present several examples. Appendix \ref{sec:Henon} is dedicated to introducing the monodromy action conjecture for the complex H\'enon family and explaining how Corollary \ref{cor:main} gives rise to a solution to the degenerate case of the conjecture. 

\medskip

\noindent
\textit{Note.}
During the preparation of this manuscript we were made aware of the preprint~\cite{BB} which addresses a similar problem under the assumption of \emph{narrowness} of a hyperbolic component. In this respect, their main result (Theorem 1.3 in~\cite{BB}) seems to be a special case of Corollary \ref{cor:main} in this manuscript.

\section{Pseudo-monodromy}\label{sec:monodromy}

In this section we first recall some basic facts on the Mandelbrot set $\mathfrak{M}$ (see~\cite{D, DH2} for details) and then formulate the notion of pseudo-monodromy for $\mathfrak{M}$. 

Write $\mathbb{T}\equiv\mathbb{R}/\mathbb{Z}$. Let $H$ be a hyperbolic component of $\mathfrak{M}$, and let $r_H$ be its root point. Let $R_{\mathfrak{M}}(\theta)$ be the external ray of angle $\theta$ for $\mathfrak{M}$ and set
\[\Theta_{\mathfrak{M}}(H)\equiv \big\{\theta\in\mathbb{T} : R_{\mathfrak{M}}(\theta) \mbox{ lands on }r_H\big\}.\]
Denote by $H_{\heartsuit}$ the \emph{Main Cardioid} of $\mathfrak{M}$. A celebrated theorem of Douady and Hubbard says 

\begin{thm}[Douady--Hubbard~\cite{DH2}]\label{thm:DH}
For any hyperbolic component $H$ different from $H_{\heartsuit}$, there are exactly two angles $0<\theta^-_{H}<\theta^+_{H}<1$ so that $\Theta_{\mathfrak{M}}(H)=\{\theta^-_{H}, \theta^+_{H}\}$.
\end{thm}


The above theorem in particular yields that the union $R_{\mathfrak{M}}(\theta_{H}^-)\cup\{r_H\}\cup R_{\mathfrak{M}}(\theta_{H}^+)$ divides the complex plane  $\mathbb{C}$ into two parts when $H$ is different from $H_{\heartsuit}$. 
The connected component of $\mathbb{C}\setminus (R_{\mathfrak{M}}(\theta_{H}^-)\cup\{r_H\}\cup R_{\mathfrak{M}}(\theta_{H}^+))$ containing $H$ is called the \emph{wake} associated with $H$ and denoted by $\mathcal{W}_{H}$. 

Given a parameter $c\in \mathbb{C}\setminus \mathfrak{M}$, there exists a unique $\theta_c\in \mathbb{T}$ so that $c\in R_{\mathfrak{M}}(\theta_c)$. Denote by $R_c(\theta)$ the dynamical ray of angle $\theta$ for the quadratic map $p_c$. Then, the union $R_c(\frac{\theta_c}{2})\cup \{0\}\cup R_c(\frac{\theta_c}{2}+\frac{1}{2})$ divides $\mathbb{C}$ into two pieces, defining a partition of $J_c$. We label the piece containing $R_c(\frac{\theta_c}{2}+\frac{3}{4})$ by $A$ and the piece containing $R_c(\frac{\theta_c}{2}+\frac{1}{4})$ by $B$. This gives a coding map: 
\[h_c : J_c \longrightarrow \big\{A, B\big\}^{\mathbb{N}}\]
which conjugates $p_c$ on $J_c$ and the shift map on $\{A, B\}^{\mathbb{N}}$. Moreover, since the curve depends continuously on $c\in \mathbb{C}\setminus \mathfrak{M}$, we can choose the coding map so that $\mathbb{C}\setminus \mathfrak{M}\ni c \mapsto h_c^{-1}(\underline{i})\in \mathbb{C}$ is continuous for every fixed $\underline{i}\in\{A, B\}^{\mathbb{N}}$.

Let $H$ be a hyperbolic component of $\mathfrak{M}$ different from the Main Cardioid and let $\theta_H^-<\theta_H^+$ be the two external angles as in Theorem \ref{thm:DH}. Consider the equipotential curve $E$ for $\mathfrak{M}$ containing the base-point $c_{\ast}\approx -2.3$. We concatenate a part of $E$ from $c_{\ast}$ to $E\cap R_{\mathfrak{M}}(\theta_H^+)$ (resp. $E\cap R_{\mathfrak{M}}(\theta_H^-)$) and the bounded connected component of $R_{\mathfrak{M}}(\theta_H^+)\setminus E$ (resp. $R_{\mathfrak{M}}(\theta_H^-)\setminus E$) to define a path $\gamma_H^+$ (resp. $\gamma_H^-$) from $c_{\ast}$ to $r_H$ in $\mathbb{C}\setminus \mathfrak{M}$ so that the concatenation: 
\[\gamma_H^+\cdot\{r_H\}\cdot (\gamma_H^-)^{-1}\] 
does not surround $H_{\heartsuit}$. An example of a pseudo-loop $\gamma$ is given in Figure \ref{fig:loop}.

\begin{figure}
\centering

\begin{tikzpicture}
\node(ex31) at (0,0){\includegraphics[width=3in]{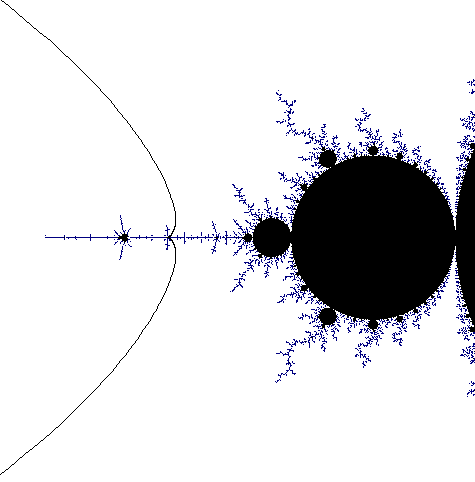}};
\draw[->,black,thin] (-3.79,3.85) .. controls (-5,4) and (-5,-4) .. (-3.78,-3.76);
\node(1) at (-5,0) {$\gamma$};
\filldraw [black] (-4.7,0) circle (1pt);
\end{tikzpicture}
  
\caption{A pseudo-loop $\gamma$ associated to the Lobster hyperbolic component.}
\label{fig:loop}
\end{figure}

Below we interpret this closed loop as if it were completely contained in the horseshoe locus for the quadratic family and consider its monodromy action. However, when $c$ varies along $\gamma_H^{\pm}$ to $r_H$, the Julia set $J_c$ implodes at  $c=r_H$ (see~\cite{D}). Hence it is not clear how the Julia set along $\gamma_H^+$ is continued to the Julia set along $\gamma_H^-$ through $c=r_H$. Our strategy is to, instead of looking at the Julia set, focus on the bifurcation of external rays as in~\cite{A1, A2}. 

To achieve this, we introduce the following symbolic coding. Let $\sigma : \mathbb{T}\to\mathbb{T}$ be the angle-doubling map. Given an angle $\alpha\in\mathbb{T}$, define a coding map
\[I^{\circ}_{\alpha} \, : \, \mathbb{T}\ni\theta\longmapsto i_0i_1\cdots\in\big\{A, B, \circ \big\}^{\mathbb{N}}\]
by 
\[i_n\equiv
\begin{cases}
	A & \mbox{ if } \sigma^n(\theta)\in \left(\frac{\alpha+1}{2}, \frac{\alpha}{2}\right) \\
	B & \mbox{ if } \sigma^n(\theta)\in \left(\frac{\alpha}{2}, \frac{\alpha+1}{2}\right) \\
	\, \circ & \mbox{ if } \sigma^n(\theta)=\frac{\alpha}{2} \mbox{ or } \frac{\alpha+1}{2}
\end{cases}
\]
for $n\geq 0$, where $\mathbb{N}$ denotes the set of natural numbers including $0$. We call $I^{\circ}_{\alpha}(\alpha)$ the \emph{kneading sequence} of the angle $\alpha$.

We say that an external ray $R_c(\theta)$ is \emph{unbranched} if it does not hit $\bigcup_{n\geq 0}p_c^{-n}(0)$. In this case, the external ray $R_c(\theta)$ lands on a point in $J_c$. One can see that an external ray $R_c(\theta)$ is unbranched if and only if $I^{\circ}_{\theta_c}(\theta)$ does not contain $\circ$. The next result indicates that the coding $h_c$ of the Cantor Julia set is closely related to the coding $I^{\circ}_{\theta_c}$ of the external angles.

\begin{thm}[Atela~\cite{A2}]
Let $R_c(\theta)$ be an unbranched external ray and let $z\in J_c$ be its landing point. Then, we have $I^{\circ}_{\theta_c}(\theta)=h_c(z)$.
\end{thm}

Recall that we have defined two codings $I^{\circ}_{\theta_H^{\pm}} : \mathbb{T}\to\{A, B, \circ \}^{\mathbb{N}}$. The theorem above motivates us to introduce the \emph{discontinuity locus} for $H$ as
\[\mathrm{Disc}(H)\equiv \big\{\theta\in\mathbb{T} : I^{\circ}_{\theta_H^+}(\theta)\neq I^{\circ}_{\theta_H^-}(\theta)\big\}.\]
Note that if $r_H$ belonged to $\mathbb{C}\setminus \mathfrak{M}$ so that $\mathfrak{M}$ were separated at $r_H$ into two parts, then the loop $\gamma_H^+\cdot\{r_H\}\cdot (\gamma_H^-)^{-1}$ would be entirely contained in $\mathbb{C}\setminus \mathfrak{M}$ and $\mathrm{Disc}(H)$ would represent the monodromy action along the loop. Therefore, we interpret the locus $\mathrm{Disc}(H)$ as the ``pseudo-monodromy'' along the ``pseudo-loop'' $\gamma_H^+\cdot\{r_H\}\cdot (\gamma_H^-)^{-1}$ in $\mathbb{C}\setminus \mathfrak{M}$. More generally, a pseudo-loop will be any loop in $\mathbb{C}\setminus \mathfrak{M}\cup\{r_H\}$ homotopic to $\gamma_H^+\cdot\{r_H\}\cdot (\gamma_H^-)^{-1}$ for some $H$. 

In order to express the difference of the two codings  $I^{\circ}_{\theta_H^+}$ and $I^{\circ}_{\theta_H^-}$, we introduce the following coding. Given a hyperbolic component $H$ we define
\[I_H \, : \, \mathbb{T}\ni\theta\longmapsto i_0i_1\cdots\in\big\{A, B, \star\big\}^{\mathbb{N}}\]
by 
\begin{equation*}
i_n =  
\begin{cases} 
A & \mbox{ if } \sigma^n(\theta) \in \left(\frac{\theta_H^+ +1}{2}, \frac{\theta_H^-}{2}\right) \\
B & \mbox{ if } \sigma^n(\theta) \in \left(\frac{\theta_H^+}{2}, \frac{\theta_H^- +1}{2}\right) \\
\, \star & \mbox{ if } \sigma^n(\theta) \in \left[\frac{\theta_H^-}{2}, \frac{\theta_H^+}{2}\right]\cup \left[\frac{\theta_H^-+1}{2}, \frac{\theta_H^++1}{2}\right]
\end{cases}
\end{equation*}
for $n\geq 0$. 

\begin{dfn}
For a finite word $\underline{w}$ of length $n$ over $\{A, B, \star \}$, denote by $\mathbb{T}_{\underline{w}}$ the set of angles $\theta\in\mathbb{T}$ whose itineraries $I_H(\theta)=i_0i_1\cdots$ satisfy $i_0\cdots i_{n-1}=\underline{w}$.
\end{dfn}

Let $\Pi_1(H)$ be the closed interval $[\theta_H^-, \theta_H^+] \subset \mathbb{T}$ and let $\Pi_0(H)\equiv \sigma^{-1}(\Pi_1(H))$. Note that we have $\Pi_0(H)=\mathbb{T}_{\star}$. Note also that $\Pi_0(H)$ is the locus where the partitions for the two codings $I^{\circ}_{\theta_H^+}$ and $I^{\circ}_{\theta_H^-}$ differ. 

We can relate the discontinuity locus to $\Pi_1(H)$ as follows.

\begin{lmm}\label{lmm:disc}
Let $H$ be a hyperbolic component of $\mathfrak{M}$. Then we have
\[\mathrm{Disc}(H)=\bigcup_{m\geq 1}\sigma^{-m}(\Pi_1(H)).\]
\end{lmm}

\begin{proof} 
We have $\theta\in\mathrm{Disc}(H)$ if and only if $I_H(\theta)$ contains $\star$, that is, if and only if $\sigma^k(\theta) \in\mathbb{T}_{\star} =\Pi_0(H)$ for some $k\geq 0$.
\end{proof}

\section{Conspicuousness}\label{sec:kneading}

In this section we explain a more detailed combinatorial structure of the Mandelbrot set $\mathfrak{M}$ and introduce the notion of conspicuousness of hyperbolic components of $\mathfrak{M}$.

One can associate the notion of a kneading sequence to each hyperbolic component~\cite{LS} the idea of which goes back to Milnor--Thurston's kneading theory for maps of the interval~\cite{MT}. For this purpose, we modify the coding $I^{\circ}_{\alpha}$ in two ways as
\[I^{\pm}_{\alpha} \, : \, \mathbb{T}\ni\theta\longmapsto i^{\pm}_0i^{\pm}_1\cdots\in\big\{A, B\big\}^{\mathbb{N}}\]
with 
\[i^+_n\equiv 
\begin{cases}
A & \mbox{ if } \sigma^n(\theta)\in \left[\frac{\alpha+1}{2}, \frac{\alpha}{2}\right) \\
B & \mbox{ if } \sigma^n(\theta)\in \left[\frac{\alpha}{2}, \frac{\alpha+1}{2}\right)
\end{cases}
\]
and 
\[i^-_n\equiv 
\begin{cases}
A & \mbox{ if } \sigma^n(\theta)\in \left(\frac{\alpha+1}{2}, \frac{\alpha}{2}\right] \\
B & \mbox{ if } \sigma^n(\theta)\in \left(\frac{\alpha}{2}, \frac{\alpha+1}{2}\right]
\end{cases}
\]
for $n\geq 0$. 

Now, let $H$ be a hyperbolic component of $\mathfrak{M}$. The period of the unique attractive cycle of $p_c$ for $c\in H$ is called the \emph{period} of $H$ and is denoted by $\mathrm{per}(H)$. Write $I^+_H\equiv I_{\theta_{H}^-}^+$ and $I^-_H\equiv I_{\theta_{H}^+}^-$. Then, it is easy to see that $I_H^+(\theta_{H}^-)=I_H^-(\theta_{H}^+)$.

\begin{dfn}
The \emph{kneading sequence} of $H$ different from $H_{\heartsuit}$ is the first $\mathrm{per}(H)$ letters of the sequence $I_H^+(\theta_{H}^-)=I_H^-(\theta_{H}^+)$ and is denoted by $K(H)$.\footnote{Our definition of $K(H)$ is the first $\mathrm{per}(H)$ entries of the kneading sequence $K(\mathcal{A})$ in~\cite{LS, S}. It is defined as the kneading sequence of the corresponding wake and denoted by $K(\mathcal{W}_H)$ in~\cite{Li, I}.} 
The \emph{discarded kneading sequence} of $H$ is the first $\mathrm{per}(H)-1$ letters of the sequence $I_H^+(\theta_{H}^-)=I_H^-(\theta_{H}^+)$ and is denoted by $\widehat{K}(H)$. 
\end{dfn}

Below we present a list of examples:
\begin{itemize}
\item When $H$ is the Basilica component whose wake is given by $\theta^-_{H}=1/3$ and $\theta^+_{H}=2/3$, we have $\mathrm{per}(H)=2$, $K(H)=BA$ and $\widehat{K}(H)=B$.
\item When $H$ is the Rabbit component whose wake is given by $\theta^-_{H}=1/7$ and $\theta^+_{H}=2/7$, we have $\mathrm{per}(H)=3$, $K(H)=BBA$ and $\widehat{K}(H)=BB$.
\item When $H$ is the Airplane component whose wake is given by $\theta^-_{H}=3/7$ and $\theta^+_{H}=4/7$, we have $\mathrm{per}(H)=3$, $K(H)=BAA$ and $\widehat{K}(H)=BA$.
\item When $H$ is the hyperbolic component whose wake is given by $\theta^-_{H}=6/15$ and $\theta^+_{H}=9/15$, we have $\mathrm{per}(H)=4$, $K(H)=BABB$ and $\widehat{K}(H)=BAB$.
\item When $H$ is the hyperbolic component whose wake is given by $\theta^-_{H}=13/31$ and $\theta^+_{H}=18/31$, we have $\mathrm{per}(H)=5$, $K(H)=BABBA$ and $\widehat{K}(H)=BABB$.
\end{itemize}

Another consequence of Theorem \ref{thm:DH} is that the Mandelbrot set $\mathfrak{M}$ has a tree-like structure. More precisely, it yields that either $\mathcal{W}_{H}\supset\mathcal{W}_{H'}$, $\mathcal{W}_{H}\subset\mathcal{W}_{H'}$ or $\mathcal{W}_{H}\cap\mathcal{W}_{H'}=\emptyset$ holds for any two hyperbolic components $H$ and $H'$ of $\mathfrak{M}$. 

For hyperbolic components $H$ and $H'$ of the Mandelbrot set, we write $H'\succ H$ if $\mathcal{W}_{H'}\subset\mathcal{W}_{H}$. For $H$ and $H'$ with $H'\succ H$, the \emph{combinatorial arc} from $H$ to $H'$ denoted by $[H, H']$ is defined as the collection of hyperbolic components $H''$ satisfying $H'\succ H''\succ H$ together with the natural ordering defined by $\succ$.

\begin{dfn}[Lipa~\cite{Li}]\label{dfn:conspicuous}
Let $H$ and $H'$ be two hyperbolic components of $\mathfrak{M}$. We say that $H'$ is \emph{conspicuous} to $H$ and denote as $H'\triangleright H$ if
\begin{enumerate}
\item $H'\succ H$,
\item $\mathrm{per}(H') < \mathrm{per}(H)$ except for the case $H=H'$,
\footnote{We can exclude the equality from the original definition in~\cite{Li} thanks to Lemma \ref{lmm:lavaurs}.}
\item there are no hyperbolic components $H''$ with $H'' \in[H, H']$ and $\mathrm{per}(H'')<\mathrm{per}(H')$.
\end{enumerate}
\end{dfn}

Note that a hyperbolic component is conspicuous to itself. The conspicuousness is transitive, i.e.  $H'\triangleright H$ and $H''\triangleright H'$ implies that $H''\triangleright H$. It immediately follows from the definition above that the periods of conspicuous components are monotone decreasing along $[H, H']$. Moreover, it is \emph{strictly} monotone by the following lemma of Lavaurs~\cite{La} (see also Lemma 3.8 of~\cite{LS}) which will be useful in the rest of this article.

\begin{lmm}[Lavaurs]\label{lmm:lavaurs}
Let $H'\succ H$ be distinct hyperbolic components with $\mathrm{per}(H)=\mathrm{per}(H')$. Then, there exists a hyperbolic component $H''\in [H, H']$ with $\mathrm{per}(H'')<\mathrm{per}(H)=\mathrm{per}(H')$. In particular, $H'$ is not conspicuous to $H$.
\end{lmm}

\section{Statement of results}\label{sec:main}

Recall that $K(H)$ denotes the kneading sequence of $H$ and $\widehat{K}(H)$ denotes the discarded kneading sequence, i.e. the word obtained by deleting the final digit of $K(H)$. Hence $\widehat{K}(H)\,\star$ is the word obtained by replacing the final digit of $K(H)$ by $\star$. Recall also that for a word $\underline{w}$ over $\{A, B, \star\}$, we denote by $\mathbb{T}_{\underline{w}}$ the set of angles $\theta\in\mathbb{T}$ whose coding by $I_H$ starts with $\underline{w}$.

Let us set
\[\Xi(H)\equiv\bigcup_{n\geq 1}\sigma^{-n}(\Theta_{\mathfrak{M}}(H)),\]
where $\Theta_{\mathfrak{M}}(H)=\{\theta^+_H, \theta^-_H\}$. Now we may present the main result of this paper.

\begin{thm}\label{thm:main}
Let $H$ be a hyperbolic component of $\mathfrak{M}$ different from $H_{\heartsuit}$. Then, 
\[\Pi_1(H)\setminus \Xi(H) \ \subset \bigcup_{H'\triangleright H}\left(\mathbb{T}_{K(H')}\cup \mathbb{T}_{\widehat{K}(H')\, \star}\right).\]
\end{thm}

For a hyperbolic component $H$, let $H_1, \cdots, H_L$ be the collection of hyperbolic components conspicuous to $H$ including $H$ itself. Set $\Sigma_L\equiv \{1, \cdots, L\}$. As a consequence of Theorem~\ref{thm:main}, we obtain the following algorithm on the monodromy action. 

\begin{cor}\label{cor:main}
Let $H$ be a hyperbolic component of $\mathfrak{M}$ different from $H_{\heartsuit}$. Then, the following (i) and (ii) are equivalent for $\theta\in\mathbb{T}\setminus \Xi(H)$.
\begin{enumerate}
\renewcommand{\theenumi}{\roman{enumi}}
\item $\theta$ belongs to $\mathrm{Disc}(H)$.
\item Both $I^+_H(\theta)$ and $I^-_H(\theta)$ agree, except when one sees, either 
\begin{enumerate}
\item words of the form $\ast \, \widehat{K}(H_{i_1})\ast \widehat{K}(H_{i_2})\ast \cdots\ast \widehat{K}(H_{i_{n-1}})\ast K(H_{i_n})$ for some $n\geq 1$, or 
\item an infinite sequence of the form $\ast \, \widehat{K}(H_{i_1})\ast \widehat{K}(H_{i_2})\ast\cdots$ 
\end{enumerate}
for some $i_1i_2\cdots\in \Sigma_L^{\mathbb{N}}$, where $\ast$ is either $A$ or $B$. Moreover, the sequences $I^+_H(\theta)$ and $I^-_H(\theta)$ have opposite letters at the $\ast$ entries.
\end{enumerate}
\end{cor}

Here, we define the \emph{opposite letter} of $A$ (resp. $B$) to be $B$ (resp. $A$) and write $\overline{A}=B$ (resp. $\overline{B}=A$). The difference at each $\ast$ entry describes the pseudo-monodromy action along a pseudo-loop. This can be thought of as describing a map from $\{A,B\}^\mathbb{N}$ to itself. Such a map commutes with the shift but fails to be injective, thus it does not describe an automorphism of the one-sided $2$-shift. We note that Atela~\cite{A1, A2} has previously obtained a similar result under the assumption that the parameter rays land at ``main bifurcation points.'' 

\begin{rmk}
It is necessary to remove $\Xi(H)$ from $\Pi_1(H)$ in Theorem \ref{thm:main} (see Remark \ref{rmk:flip} in Appendix \ref{sec:examples}). However, when $\sigma^k\left(\theta_H^{\pm}\right)$ is disjoint from $\Pi_0(H)$ for $0\le k<\mathrm{per}(H)-1$, we have
\[\Pi_1(H)\ \subset \bigcup_{H'\triangleright H}\left(\mathbb{T}_{K(H')}\cup \mathbb{T}_{\widehat{K}(H')\, \star}\right).\]
In this case, Corollary \ref{cor:main} holds for all $\theta\in\mathbb{T}$ rather than $\theta\in\mathbb{T}\setminus \Xi(H)$.
\end{rmk}

Let us first prove Corollary \ref{cor:main} by assuming Theorem \ref{thm:main}.

\begin{proof}[Proof of Corollary \ref{cor:main}]
Since $\mathbb{T}_{\star}=\Pi_0(H)=\sigma^{-1}(\Pi_1(H))$, Theorem \ref{thm:main} yields that once an orbit belongs to $\mathbb{T}_{\star}$, its successive itinerary is either $K(H_i)$ or $\widehat{K}(H_i)\star$ for some $i$. Hence, a recursive application of Theorem~\ref{thm:main} to the $\star$ in $\bigcup_{i=1}^L \mathbb{T}_{\widehat{K}(H_i)\,\star}$ yields 
%
\begin{align*}
\Pi_1(H)\setminus \Xi(H) \subset & \bigcup_{i=1}^L \mathbb{T}_{K(H_i)}\cup \bigcup_{i=1}^L \mathbb{T}_{\widehat{K}(H_i)\,\star} \\
\subset & \bigcup_{i_1=1}^L \mathbb{T}_{K(H_{i_1})}\cup 
\bigcup_{i_1, i_2=1}^L \mathbb{T}_{\widehat{K}(H_{i_1})\,\star\, K(H_{i_2})}\cup \bigcup_{i_1, i_2=1}^L \mathbb{T}_{\widehat{K}(H_{i_1})\,\star\, \widehat{K}(H_{i_2})\,\star} \\
\subset & \cdots \\
\subset & \Bigg(\bigcup_{n\geq 1}\bigcup_{i_1 \cdots i_n\in\Sigma_L^n}\mathbb{T}_{\widehat{K}(H_{i_1})\,\star\, \cdots \,\star\, \widehat{K}(H_{i_{n-1}})\,\star \, K(H_{i_n})}\Bigg) \cup \Bigg(\bigcup_{i_1 i_2 \cdots\in \Sigma_L^{\mathbb{N}}} \mathbb{T}_{\widehat{K}(H_{i_1})\,\star\, \widehat{K}(H_{i_2})\,\star\,\cdots}\Bigg).
\end{align*}
Suppose that $\theta$ satisfies (i), i.e. $\theta\in\mathrm{Disc}(H)$. By Lemma \ref{lmm:disc}, there exists $m\geq 1$ so that $\sigma^m(\theta)\in\Pi_1(H)$. The inclusion above yields that $I^{\pm}_H(\sigma^m(\theta))$ starts with either $\widehat{K}(H_{i_1})\ast \cdots \ast \widehat{K}(H_{i_{n-1}})\ast K(H_{i_n})$ for some $n\geq 1$ or $\widehat{K}(H_{i_1})\,\ast\, \widehat{K}(H_{i_2})\,\ast\,\cdots$. Moreover, since $\mathbb{T}_{\star}=\sigma^{-1}(\Pi_1(H))$ holds, $\sigma^m(\theta)\in\Pi_1(H)$ implies that $\sigma^{m-1}(\theta)\in \mathbb{T}_{\star}$. Hence $I^{\pm}_H(\sigma^{m-1}(\theta))$ starts with either $\ast\, \widehat{K}(H_{i_1})\ast\cdots \ast \widehat{K}(H_{i_{n-1}}) \ast K(H_{i_n})$ for some $n\geq 1$ or $\ast\, \widehat{K}(H_{i_1})\ast \widehat{K}(H_{i_2})\ast\cdots$. It follows that both $I^+_H(\theta)$ and $I^-_H(\theta)$ contain either (a) or (b) in (ii). 

Conversely, suppose that $\theta$ satisfies (ii). Then there exists a digit, say the $m$-th digit, such that  $I^+_H(\theta)$ and $I^-_H(\theta)$ have opposite letters at that digit. It follows that $\sigma^m(\theta)\in \mathbb{T}_{\star}=\Pi_0(H)$, hence $\theta\in\mathrm{Disc}(H)$.
\end{proof}

\section{Proof of Theorem  \ref{thm:main}}\label{sec:proof}

In this section we prove Theorem \ref{thm:main}. For this purpose, we briefly recall Milnor's orbit portrait theory~\cite{M} and Thurston's invariant lamination theory~\cite{T}. Let $\mathbb{D}\equiv \{z\in\mathbb{C} : |z|<1\}$ be the open unit disk equipped with the Poincar\'e metric. We identify $\mathbb{T}\cong \partial\mathbb{D}$ by the map $\theta \mapsto e^{2\pi i\theta}$. A \emph{leaf} $\ell=\ell_{\left\{\alpha, \beta\right\}}$ with endpoints $\alpha, \beta\in \mathbb{T}$ is a curve between $\alpha$ and $\beta$ so that $\ell\cap\mathbb{D}$ is a hyperbolic geodesic in $\mathbb{D}$. Its length is defined to be the distance along $\mathbb{T}$ between $\alpha$ and $\beta$. The angle-doubling map $\sigma$ acts on a leaf as $\sigma(\ell_{\left\{\alpha, \beta\right\}})\equiv \ell_{\left\{\sigma(\alpha), \sigma(\beta)\right\}}$.

Fix a hyperbolic component $H$ as in Theorem \ref{thm:main}. Let us write $N\equiv\mathrm{per}(H)$ and $K(H)=k_1\cdots k_N$. 

\begin{lmm}\label{lmm:overlap}
Let $H'\triangleright H$ be distinct hyperbolic components. Then,
\begin{enumerate}
\renewcommand{\theenumi}{\roman{enumi}}
\item the first string of length $\mathrm{per}(H')-1$ in $K(H')$ is identical to the one in $K(H)$,
\item the $\mathrm{per}(H')$-th letter of $K(H')$ is opposite to the one of $K(H)$.
\end{enumerate}
\end{lmm}

\begin{proof}
We will show that the finite kneading sequence of an angle on $\mathbb{T}$ changes only when crossing a point of period $n$, and it changes in the $n$-th position. Thus, since periods of conspicuous components $H'\triangleright H$ are monotone decreasing from $H$ to the extremity of $\mathfrak{M}$ we have for $H'$ that the first string of length $\mathrm{per}(H')-1$ in $K(H')$ is identical to the one in $K(H)$, obtaining (i). Since the difference is a flipping of the symbol in the $\mathrm{per}(H')$-th entry, we then obtain (ii).

Suppose that $\theta$ is periodic with period $n$ under $\sigma$. For $\varepsilon>0$ sufficiently small, the first $n-1$ entries of $I^+_\theta(\theta+\varepsilon)$ and $I^+_\theta(\theta-\varepsilon)$ (resp. $I^-_\theta(\theta+\varepsilon)$ and $I^-_\theta(\theta-\varepsilon)$) agree. Since $\theta$ has period $n$, $\sigma^{n-1}(\theta)$ is either $\frac{\theta}{2}$ or $\frac{\theta+1}{2}$. Hence $\sigma^{n-1}(\theta + \varepsilon)$ and $\sigma^{n-1}(\theta - \varepsilon)$ lie either side of the leaf $\ell_{\left\{\frac{\theta}{2}, \frac{\theta+1}{2} \right\}}$ and the distance between them is $2^n\varepsilon$. The distance between the leaves $\ell_{\left\{\frac{\theta-\varepsilon}{2}, \frac{\theta+1-\varepsilon}{2} \right\}}$ and $\ell_{\left\{\frac{\theta+\varepsilon}{2}, \frac{\theta+1+\varepsilon}{2} \right\}}$ is $\varepsilon$, thus the endpoints of both leaves lie between $\sigma^{n-1}(\theta + \varepsilon)$ and $\sigma^{n-1}(\theta - \varepsilon)$. Since these leaves define the partition for the kneading sequence of $\theta\pm\varepsilon$, we conclude that the $n$-th entry of the kneading sequences of $\theta\pm\varepsilon$ have opposite symbol and earlier entries agree. This proves the claim.
\end{proof}

Based on this lemma, we introduce the following notion.

\begin{dfn}
We say that $n$ with $1\leq n\leq N-1$ is a \emph{return time} for $H$ if $n=\mathrm{per}(H')$ for some conspicuous component $H'\triangleright H$ different from $H$.
\end{dfn}

Although $N=\mathrm{per}(H)$, we do not call $N$ a return time for $H$.


Let us continue to recall necessary ingredients from~\cite{M, T}. Define the \emph{minor leaf} (or the \emph{characteristic arc}) of $H$ to be the geodesic in $\overline{\mathbb{D}}$ connecting $\theta_H^-$ and $\theta_H^+$ and denoted by $\ell_1$. For $1< n\leq N$ we then inductively define a leaf $\ell_n$ to be the unique geodesic in $\overline{\mathbb{D}}$ connecting the $\sigma$-images of the endpoints of $\ell_{n-1}$. The \emph{major leaves} $\ell_0$ and $\ell_0'$ are the unique two geodesics in $\overline{\mathbb{D}}$ connecting the preimages of the endpoints of $\ell_1$ by $\sigma$ so that $\ell_0= \ell_N$. Set $\mathcal{L}\equiv \{\ell_0', \ell_0, \ell_1, \cdots, \ell_{N-1}\}$. Since the dynamical external rays corresponding to the two endpoints of a leaf in $\mathcal{L}$ land on a same point in the Julia set, distinct elements in $\mathcal{L}$ do not intersect transversally with each other but they may intersect at endpoints.  

\begin{rmk}\label{rmk:period}
The period of the endpoints of every leaf in $\mathcal{L}$ is $N$, but the minimal period of a leaf in $\mathcal{L}$ as a set is either $N$ or $N/2$. In the latter case, every leaf is invariant but flipped by $\sigma^{N/2}$ (see, e.g., Figure \ref{fig:ex2} in  Appendix \ref{sec:examples}). In particular, since the angle-doubling map $\sigma$ is orientation preserving, the $\Pi_0(H)$-side of an endpoint of $\ell_0$ is mapped to the same side of the same endpoint by $\sigma^N$ and to the non-$\Pi_0(H)$-side of the other endpoint by $\sigma^{N/2}$ (See Figure \ref{fig:ex2}). 
\end{rmk}

\begin{lmm}\label{lmm:outside}
All leaves in $\mathcal{L}$ are disjoint from $\mathrm{int}(\Pi_0(H))$.
\end{lmm}

\begin{proof}
We know that the length of the minor leaf is the shortest among the leaves in $\mathcal{L}$, and hence the major leaves are the longest among the leaves in $\mathcal{L}$ (see Lemma 2.6 in~\cite{M}). This yields that there are no leaves in $\mathcal{L}$ except for $\ell_0$ and $\ell_0'$ which connects the two connected components of $\Pi_0(H)$. Since the length of the minor leaf is twice the length of a connected component of $\Pi_0(H)$, there are no leaves in $\mathcal{L}$ whose endpoints belong to the same connected component of $\Pi_0(H)$. This proves the claim.
\end{proof}

Set $R_0\equiv \Pi_0(H)$. For $0\leq n\leq N-1$, we set
\begin{equation*}
Q_n \equiv 
\begin{cases}
R_n & \mbox{ if } n \mbox{ is not a return time} \\
R_n\cap \overline{\mathbb{T}_{k_n}} & \mbox{ if } n \mbox{ is a return time}
\end{cases}
\end{equation*}
and inductively define $R_{n+1}\equiv\sigma(Q_n)$. Note that $R_n$ (resp. $Q_n$) is the union of finitely many closed intervals in the circle. Let $\mathrm{Poly}(R_n)$ (resp. $\mathrm{Poly}(Q_n)$) be the ``polygon'' including its interior whose edges consist of the connected components of $R_n$ (resp. $Q_n$) and the leaves in $\mathcal{L}$ whose endpoints are endpoints of neighbouring connected components of $R_n$ (resp. $Q_n$). By the definition of $R_n$ (resp. $Q_n$) and the connectivity of $R_0$, we inductively see that $\mathrm{Poly}(R_n)$ (resp. $\mathrm{Poly}(Q_n)$) is connected.

\begin{lmm}\label{lmm:intersect}
Let $K(H)=k_1\cdots k_N$. Then, we have the following.
\begin{enumerate}
\renewcommand{\theenumi}{\roman{enumi}}
\item The endpoints of $\ell_n$ belong to $\overline{\mathbb{T}_{k_n}}$ and $\ell_n$ is an edge of $\mathrm{Poly}(Q_n)$ for $1\leq n\leq N-1$. 
\item $R_n\cap\mathbb{T}_{k_n}\ne \emptyset$ for $1\leq n\leq N$ and $R_N\supset \Pi_0(H)$.
\end{enumerate}
\end{lmm}

\begin{proof}
(i) Since the $A$-piece (resp. $B$-piece) for the $I^{\pm}_H$-partitions is contained in the union $\mathbb{T}_A\cup\mathbb{T}_{\star}$ (resp. $\mathbb{T}_B\cup\mathbb{T}_{\star}$), we see that the endpoints of $\ell_n$ belong to $\mathbb{T}_{k_n}\cup\mathbb{T}_{\star}$ for $1\leq n\leq N$. The endpoints of $\ell_n$ cannot belong to $\mathrm{int}(\mathbb{T}_{\star})$ by Lemma \ref{lmm:outside}, therefore either (a) they belong to $\overline{\mathbb{T}_{k_n}}$ or (b) $\ell_n=\ell_0$ with possibly interchanging the endpoints. Note that these two conditions are not contrary to each other (see Remark \ref{rmk:flip}).

We show that the condition (a) is always satisfied for $1\leq n\leq N-1$. To see this, suppose first that $\ell_n=\ell_0$ holds with fixed endpoints. This contradicts to the minimality of the period $N$. Suppose next that $\ell_n=\ell_0$ holds with interchanged endpoints. Then, one sees from the definition of the $I^{\pm}_H$-partition that the endpoints of $\ell_n$ belong to $\overline{\mathbb{T}_{k_n}}$, hence (a) is satisfied. This proves the first half of (i). The second half of (i) follows from this and the definition of $Q_n$. 

(ii) By the definition of $K(H)$, we see that both $\sigma^{n-1}(\theta^+_H-\varepsilon)$ and $\sigma^{n-1}(\theta^-_H+\varepsilon)$ belong to $\mathbb{T}_{k_n}$ for $1\leq n\leq N-1$ and a small $\varepsilon>0$. Note that when $\ell_n=\ell_0$ holds with endpoints interchanged, the $\Pi_0(H)$-side of an endpoint of $\ell_0$ is mapped to the opposite side of the other endpoint by $\sigma^n$. Since $\sigma^{n-1}(\theta^+_H-\varepsilon)$ and $\sigma^{n-1}(\theta^-_H+\varepsilon)$ also belong to $R_n$ as well by (i) and Remark \ref{rmk:period}, this proves $R_n\cap\mathbb{T}_{k_n}\ne \emptyset$ for $1\leq n\leq N-1$.

For the case $n=N$, we first note that $\ell_0$ is an edge of $\mathrm{Poly}(R_N)$ by (i). Since the $\Pi_0(H)$-side of an endpoint of $\ell_0$ is mapped to the same side by $\sigma^N$ as in Remark \ref{rmk:period}, Lemma \ref{lmm:outside} yields $R_N\supset \Pi_0(H)$. Suppose that $R_N=\Pi_0(H)$ holds. Then, $\ell_0'$ is an edge of $R_N$. Since every edge of $R_N$ not contained in $\mathbb{T}$ is the iterated image of either $\ell_0$ or $\ell_0'$ by $\sigma$, this implies that $\ell_0'$ is periodic, a contradiction. Therefore, $R_N\cap\mathbb{T}_{k_N}\ne \emptyset$.
\end{proof}

As a consequence of these lemmas, we have the following dichotomy.

\begin{cor}\label{cor:dichotomy}
Either $R_n\subset \overline{\mathbb{T}_{k_n}}$ or $R_n \supset \Pi_0(H)$ holds for $1\leq n\leq N$.
\end{cor}

\begin{proof}
Lemma \ref{lmm:outside} yields that either $R_n\cap \mathrm{int}(\Pi_0(H))=\emptyset$ or $R_n\supset \Pi_0(H)$ holds. Suppose that $R_n\cap \mathrm{int}(\Pi_0(H))=\emptyset$. Since $R_n\cap \mathbb{T}_{k_n}\ne\emptyset$ by Lemma \ref{lmm:intersect} (ii), the connectivity of $\mathrm{Poly}(R_n)$ yields that $R_n\subset \overline{\mathbb{T}_{k_n}}$. Therefore, the conclusion follows.
\end{proof}


This dichotomy is closely related to the notion of return time as follows.

\begin{prp}\label{prp:non-return}
If $R_n\supset \Pi_0(H)$ holds for $1\leq n\leq N-1$, then $n$ is a return time.
\end{prp}

\begin{proof}
Assume that $R_n\supset \Pi_0(H)$ holds. Since the angle-doubling mapping $\sigma$ is a surjection from $\Pi_0(H)$ to $\Pi_1(H)$, there exists a fixed point $\tilde{\theta}$ of $\sigma^n$ in $\Pi_1(H)$. Hence $\tilde{\theta}$ is periodic under $\sigma$ with period dividing $n$. Since $\tilde{\theta}$ is periodic, a result of Douady and Hubbard~\cite{DH2} (see Theorem C.7 of \cite{GM} or Proposition 3.1 of \cite{S}) says that the corresponding parameter ray $R_{\mathfrak{M}}(\tilde{\theta})$ lands on the root of a hyperbolic component $\widetilde{H}$. By Theorem \ref{thm:DH}, there exists another angle $\tilde{\theta}'$ so that the both parameter rays $R_{\mathfrak{M}}(\tilde{\theta})$ and $R_{\mathfrak{M}}(\tilde{\theta}')$ land on the root point of $\widetilde{H}$. Let $\tilde{\ell}=\ell_{\{\tilde{\theta}, \tilde{\theta}'\}}$ be the leaf in $\overline{\mathbb{D}}$ connecting $\tilde{\theta}$ and $\tilde{\theta}'$.


We claim that $\widetilde{H}$ is conspicuous to $H$. Since $\tilde{\theta}$ and $\tilde{\theta}'$ are contained in $\Pi_1(H)$, we have $\widetilde{H}\succ H$ and hence the condition (1) of Definition \ref{dfn:conspicuous} holds. Since $\mathrm{per}(\widetilde{H})\leq n<N=\mathrm{per}(H)$, the condition (2) of Definition \ref{dfn:conspicuous} holds. 

Suppose therefore that the condition (3) of Definition \ref{dfn:conspicuous} fails to hold for $\widetilde{H}$. This implies that there exists a hyperbolic component $H_1'$ in the combinatorial arc $[H, \widetilde{H}]$ with $\mathrm{per}(H_1')<\mathrm{per}(\widetilde{H})$.  Note that, by the construction of $\tilde{\ell}$, the first $\mathrm{per}(\widetilde{H})$ letters in $K(H)$ coincide with those of $K(\widetilde{H})$. Therefore, there should exist a hyperbolic component $H_1''\in [H, \widetilde{H}]$ different from $H_1'$ so that $\mathrm{per}(H_1'')=\mathrm{per}(H_1')$ (otherwise the $\mathrm{per}(H_1'')$-th letters of $K(H)$ and $K(\widetilde{H})$ disagree). 

Recall that the kneading sequence of an angle changes at the $mn$-th position for $m \in \mathbb{N}$ as we cross a periodic point of period $n$. Thus, if there exists a hyperbolic component $\widehat{H}$ such that $\widehat{H} \succ H_1'$ but $\widetilde{H} \not\succ \widehat{H}$ (this means that $\widehat{H}$ branches off from the combinatorial arc $[H, \widetilde{H}]$), then the flipping of the kneading sequence occurs and is reversed by the single branched hyperbolic component. Therefore, this situation does not affect what follows.

By applying Lemma \ref{lmm:lavaurs}, we see that there exists a hyperbolic component $H_2'$ in $[H, \widetilde{H}]$ so that $\mathrm{per}(H_2')$ is strictly smaller than that of $\mathrm{per}(H_1')$. Then, there should exist a hyperbolic component $H_2''$ in $[H, \widetilde{H}]$ different from $H_2'$ so that $\mathrm{per}(H_2'')=\mathrm{per}(H_2')$ as before. By applying this procedure repeatedly, we obtain two sequences of hyperbolic components $H_1', H_2', \cdots$ and $H_1'', H_2'', \cdots$ with $H_i'\ne H_i''$ and $\mathrm{per}(H_i')=\mathrm{per}(H_i'')$. Since the periods of these hyperbolic components are strictly decreasing in both sequences, they turn out to be finite sequences and the period of the last component is $1$ in both sequences. However, there is only one hyperbolic component of period $1$, a contradiction. Therefore, the condition (3) of Definition \ref{dfn:conspicuous} holds and $\widetilde{H}$ is conspicuous to $H$. This proves that $n$ is a return time.
\end{proof}

\begin{rmk}
We conjecture that the converse statement is also true in Proposition \ref{prp:non-return}. This would imply the equality in Theorem \ref{thm:main} modulo $\Xi(H)$. Indeed, in the cases where $H$ is only conspicuous to itself or $H$ has only one distinct conspicuous component, applying Lemma \ref{cor:dichotomy} allows us to show the converse statement.
\end{rmk}

\begin{prp}\label{prp:last}
We have $R_N\subset \mathbb{T}_{k_N}\cup \mathbb{T}_{\star}$.
\end{prp} 

\begin{proof}
By Lemma \ref{lmm:intersect} (i), $\ell_n$ is an edge of $\mathrm{Poly}(Q_n)$. In particular, $\mathrm{Poly}(Q_{N-1})$ is contained in the closure of one of the connected components of $\overline{\mathbb{D}}\setminus \ell_{N-1}$. On the other hand, we see that $Q_{N-1} \subset \overline{\mathbb{T}_{k_{N-1}}}$ by the definition of $Q_n$ and the length of $\overline{\mathbb{T}_{k_{N-1}}}$ is strictly less than $1/2$. Since $\sigma(\ell_{N-1})=\ell_0$, it follows that $\mathrm{Poly}(R_N)=\mathrm{Poly}(\sigma(Q_{N-1}))$ is contained in the closure of one of the connected components of $\overline{\mathbb{D}}\setminus \ell_0$. Since $R_N\supset \Pi_0(H)$ by Lemma \ref{lmm:intersect} (ii), we conclude that $R_N\subset \mathbb{T}_{k_N}\cup \mathbb{T}_{\star}$.
\end{proof}

\begin{proof}[Proof of Theorem \ref{thm:main}]

Suppose that we take $\theta\in\Pi_1(H)$ and consider its orbit $\sigma^n(\theta)$. When $n$ is not a return time for any conspicuous components, its coding is $k_n$ by Corollary \ref{cor:dichotomy} and Proposition \ref{prp:non-return}. When $n$ is a return time for a conspicuous component $H'$, the set of points which do not belong to $\mathbb{T}_{K(H')}\cup\mathbb{T}_{\widehat{K}(H')\star}$ is $Q_n$ and we follow its further iteration. At the final time $n=N$, all points in the remaining set $R_N$ have coding either $k_N$ or $\star$ thanks to Proposition \ref{prp:last}. Therefore,  they belong to $\mathbb{T}_{K(H)}\cup\mathbb{T}_{\widehat{K}(H)\star}$.
\end{proof}

\appendix

\section{Examples}\label{sec:examples}

In order to understand the proof of Theorem \ref{thm:main}, in this Appendix we present several hyperbolic components together with their conspicuous components and show how the sets $R_i$ and $Q_i$ behave. In the following computations, we only write the numerators over a fixed denominator of rational angles. The symbol $e$ indicates an image of an endpoint of $\Pi_1(H)$ by $\sigma$. See also pages 454 and 455 of~\cite{D}. 

\medskip

\noindent
\textit{Example 1.} Consider the hyperbolic component $H_{\mathrm{Lob}}$ with $H_{\mathrm{Air}}\triangleright H_{\mathrm{Lob}}$, where   
\begin{enumerate}
\item $\Pi_1(H_{\mathrm{Lob}})=[13/31, 18/31]$, the $\star$-piece is the union of $[13/62, 18/62]$ and $[44/62, 49/62]$, the $A$-piece is $(49/62, 13/62)$, the $B$-piece is $(18/62, 44/62)$, $K(H_{\mathrm{Lob}})=BABBA$ and $\mathrm{per}(H_{\mathrm{Lob}})=5$, 
\item $\Pi_1(H_{\mathrm{Air}})=[3/7, 4/7]$, $K(H_{\mathrm{Air}})=BAA$ and $\mathrm{per}(H_{\mathrm{Air}})=3$.
\end{enumerate}
Below, we only write the numerators over the denominator $62$. Then,
\[\Pi_1(H_{\mathrm{Lob}})=[26^e, 36^e]\mapsto [52^e, 10^e]\mapsto [42^e, 20^e]\]
whose coding is either $BAA$, $BA\,\star$ or $BAB$. The $BAB$-part of the last image is mapped as
\[[18, 20^e]\cup[42^e, 44] \mapsto [36, 40^e]\cup[22^e, 26]\mapsto [10, 18^e]\cup[44^e, 52]\]
whose coding is either $BABBA$ or $BABB\,\star$. This is illustrated in Figure \ref{fig:ex1}.

\begin{figure}
\centering
\begin{tikzpicture}

\node(ex11) at (0,0){\includegraphics[width=1.5in]{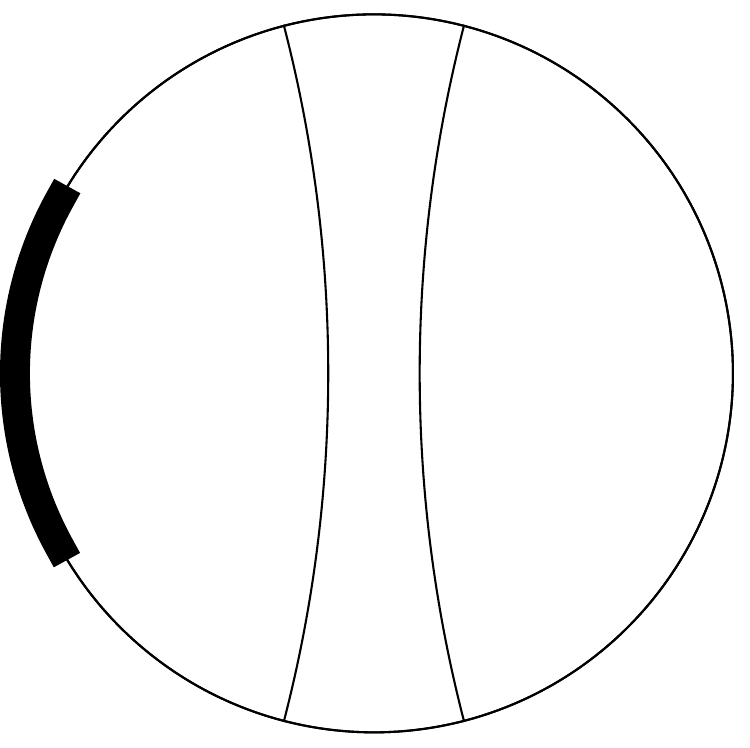}};
\node(26) at (-1.8,1.3) {26};
\node(36) at (-1.8,-1.3) {36};
\draw[->,black, thick] (2.0,0) -- (3,0);
\node[] at (2.5,0.3) {$\sigma$};
\node (ex12) at (5,0){\includegraphics[width=1.5in]{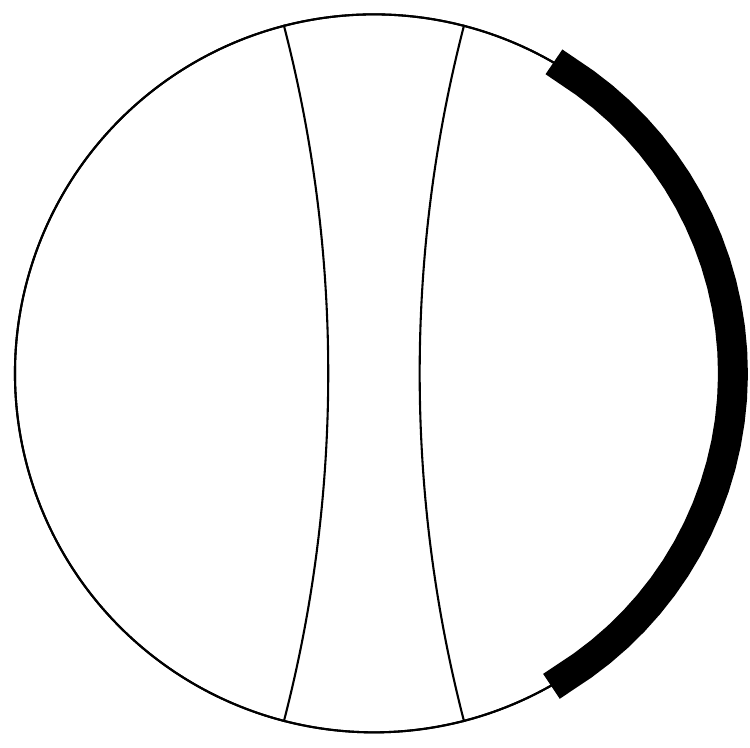}};
\node(10) at (6.2,1.8) {10};
\node(52) at (6.2,-1.8) {52};
\draw[->,black, thick] (7,0) -- (8,0);
\node[] at (7.5,0.3) {$\sigma$};
\node (ex13) at (10,0){\includegraphics[width=1.5in]{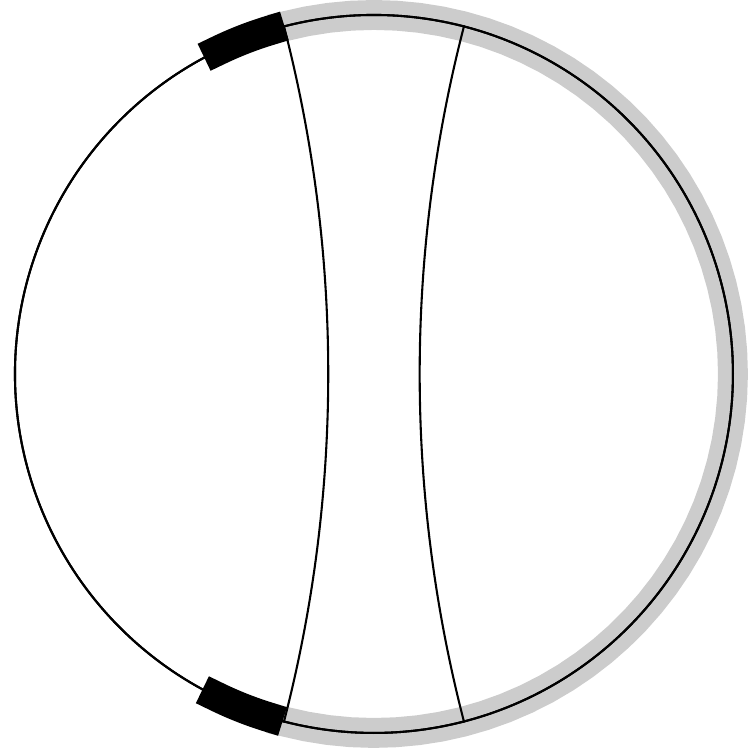}};
\node(18) at (9.5,2.1) {18};
\node(20) at (8.84,1.9) {20};
\node(42) at (8.84,-1.9) {42};
\node(44) at (9.5,-2.1) {44};
\draw[->,black,thick,rounded corners =5mm] (12,0) to (13,0) to (13,-3) to (-1,-3) to (-1,-6) to (0,-6);
\node[] at (5,-2.7) {$\sigma$};
\node(ex14) at (2.5,-6){\includegraphics[width=1.5in]{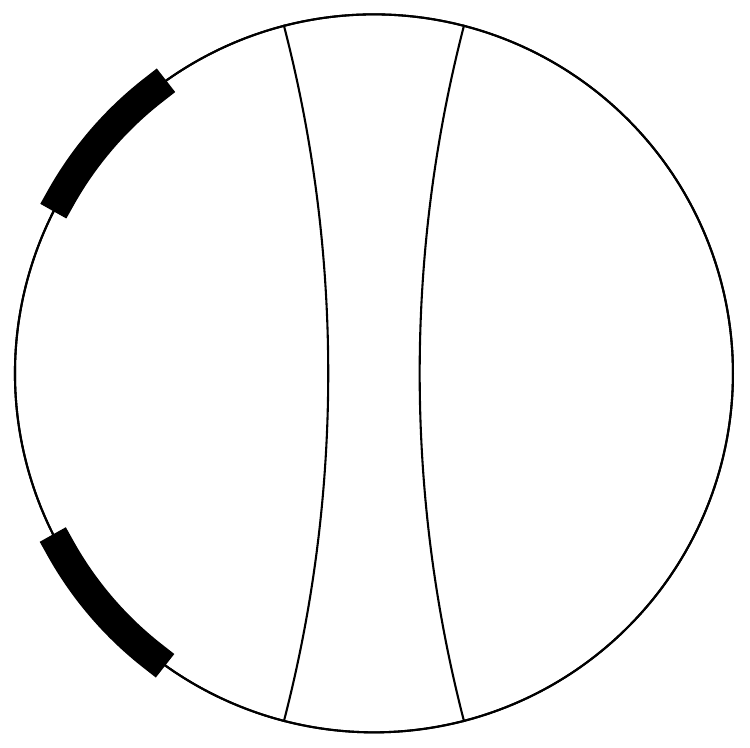}};
\node(22) at (1.4,-4.2) {22};
\node(26) at (.5,-5.2) {26};
\node(36) at (.5,-6.8) {36};
\node(40) at (1.4,-7.8) {40};

\draw[->,black, thick] (4.5,-6) -- (5.5,-6);
\node[] at (5,-5.7) {$\sigma$};
\node(ex15) at (7.5,-6){\includegraphics[width=1.5in]{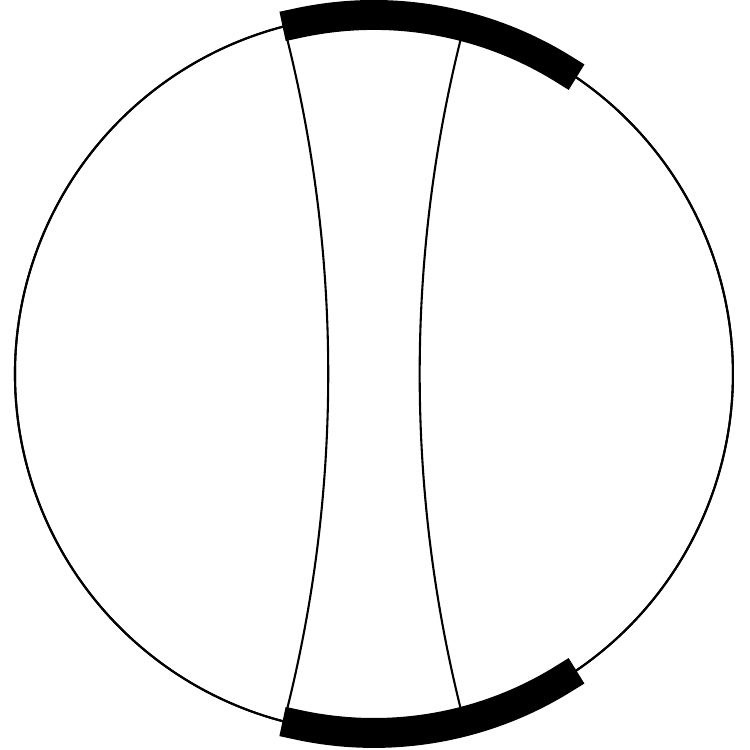}};
\node(18) at (6.9,-3.9) {18};
\node(44) at (6.9,-8.1) {44};

\node(1-) at (8.7,-4.1) {10};
\node(52) at (8.7,-7.9) {52};
    
\end{tikzpicture}

\caption{$Q_i$ (black) and $R_i$ (grey and black) for $1 \le i \le 5$ for $H_{\mathrm{Lob}}$. }
\label{fig:ex1}
\end{figure}

\medskip

\noindent
\textit{Example 2.} Consider the hyperbolic component $H_4$ with $H_{\mathrm{Air}}\triangleright H_4$, where 
\begin{enumerate}
\item $\Pi_1(H_4)=[2/5, 3/5]$, the $\star$-piece is the union $[2/10, 3/10]\cup [7/10, 8/10]$, the $A$-piece is $(8/10, 2/10)$, the $B$-piece is $(3/10, 7/10)$, $K(H_4)=BABB$ and $\mathrm{per}(H_4)=4$,
\item $\Pi_1(H_{\mathrm{Air}})=[3/7, 4/7]$, $K(H_{\mathrm{Air}})=BAA$ and $\mathrm{per}(H_{\mathrm{Air}})=3$. 
\end{enumerate}
Below, we only write the numerators over the denominator $10$. Then,
\[\mathrm{int}(\Pi_1(H_4))=(4^e, 6^e)\mapsto (8^e, 2^e)\mapsto (6^e, 4^e)\]
whose coding is either $BAA$, $BA\,\star$ or $BAB$. The $BAB$-part of the last image is mapped as
\[(3, 4^e)\cup(6^e, 7)\mapsto (6, 8^e)\cup (2^e, 4)\]
whose coding is either $BABB$ or $BAB\,\star$. This is illustrated in Figure \ref{fig:ex2}.

\begin{rmk}\label{rmk:flip}
Although the period of $H_4$ is $4$, the major leaf $\ell_0=\ell_{\{\frac{1}{5}, \frac{4}{5}\}}$ is invariant but flipped by $\sigma^2$ (see Figure \ref{fig:ex2}). Note also that we have $I_{H_4}(2/5)=I_{H_4}(3/5)=B\star B\star \,\cdots$, which does not belong to $\bigcup_{H'\triangleright H_4}(\mathbb{T}_{K(H')}\cup \mathbb{T}_{\widehat{K}(H')\, \star})$ in Theorem \ref{thm:main}. This is a reason why we need to delete $\Xi(H_4)$ from $\Pi_1(H_4)$ in Theorem \ref{thm:main}.
\end{rmk}

\begin{figure}

\centering
\begin{tikzpicture}

\node (ex21) at (5,0) {\includegraphics[width=1.5in]{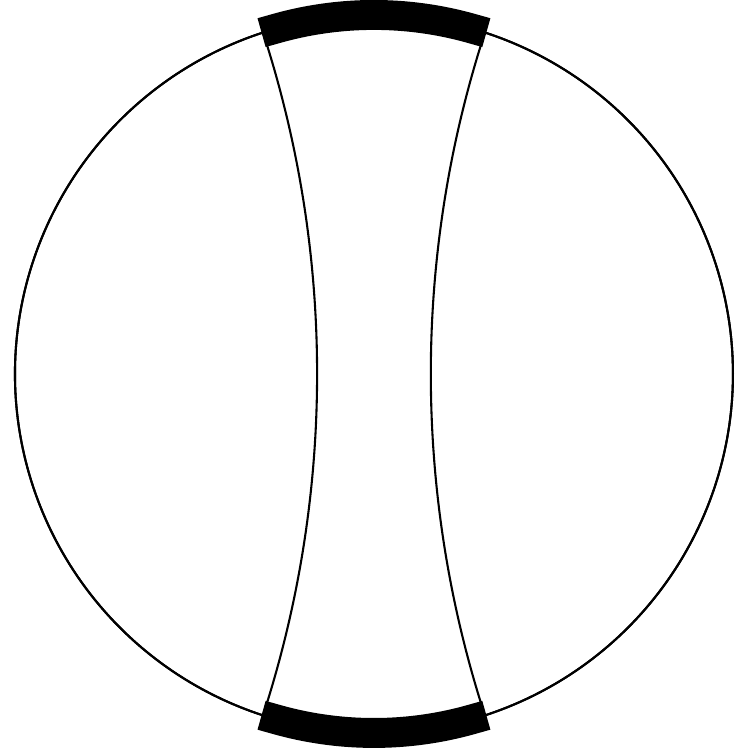}};
\draw[->,black, thick] (7,0) -- (8,0);
\node[] at (7.5,0.3) {$\sigma^2$};
\node(3) at (4.4,2.1) {3};
\node(7) at (4.4,-2.1) {7};
\node(2) at (5.6,2.1) {2};
\node(8) at (5.6,-2.1) {8};
\node (ex22) at (10,-0.1){\includegraphics[width=1.5in]{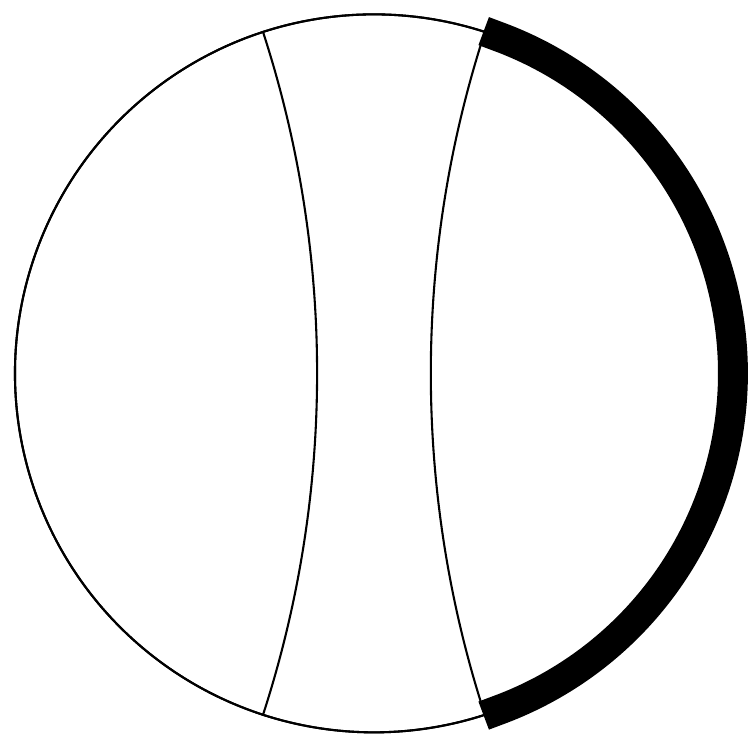}};
\node(2) at (10.6,2.1) {2};
\node(8) at (10.6,-2.1) {8};
\end{tikzpicture}
\caption{$Q_i$ and $R_i$ (both black) for $i=0$ and $i=2$ for $H_4$. The leaf on the right is $\ell_0=\ell_{\{\frac{1}{5}, \frac{4}{5}\}}$ and the leaf on the left is $\ell'_0=\ell_{\{\frac{3}{10}, \frac{7}{10}\}}$. The $\Pi_0(H_4)$-side of the point $2/10$ is mapped by $\sigma^2$ to the non-$\Pi_0(H_4)$-side of the point $8/10$.}
\label{fig:ex2}
\end{figure}

\medskip

\noindent
\textit{Example 3.} Consider the hyperbolic component $H_6$ with $H_{\mathrm{Air}}\triangleright H_5\triangleright H_6$, where  
\begin{enumerate}
\item $\Pi_1(H_6)=[26/63, 37/63]$\footnote{In page 454 of~\cite{D}, the value $26/63$ was incorrectly computed as $23/63$.}, the $\star$-piece is the union $[26/126, 37/126]\cup [89/126, 100/126]$, the $A$-piece is $(100/126, 26/126)$, the $B$-piece is $(37/126, 89/126)$, $K(H_6)=BABBBB$ and $\mathrm{per}(H_6)=6$, 
\item $\Pi_1(H_{\mathrm{Air}})=[3/7, 4/7]$, $K(H_{\mathrm{Air}})=BAA$ and $\mathrm{per}(H_{\mathrm{Air}})=3$, 
\item $\Pi_1(H_5)=[13/31, 18/31]$, $K(H_5)=BABBA$ and $\mathrm{per}(H_5)=5$.
\end{enumerate}
Below, we only write the numerators over the denominator $126$. Then,
\[\mathrm{int}(\Pi_1(H_6))=(52^e, 74^e)\mapsto (104^e, 22^e)\mapsto (82^e, 44^e)\]
whose coding is either $BAA$, $BA\,\star$ or $BAB$. The $BAB$-part of the last image is mapped as
\[(37, 44^e)\cup(82^e, 89)\mapsto (74, 88^e)\cup(38^e,52)\mapsto (22, 50^e)\cup(76^e, 104)\]
whose coding is either $BABBA$, $BABB\,\star$ or $BABBB$. The $BABBB$-part of the last image is mapped as
\[(37, 50^e)\cup(76^e, 89)\mapsto (74, 100^e)\cup (26^e, 52)\]
whose coding is either $BABBBB$ or $BABBB\,\star$. This is illustrated in Figure \ref{fig:ex3}.

\begin{figure}
\centering

\begin{tikzpicture}

\node(ex31) at (0,0){\includegraphics[width=1.5in]{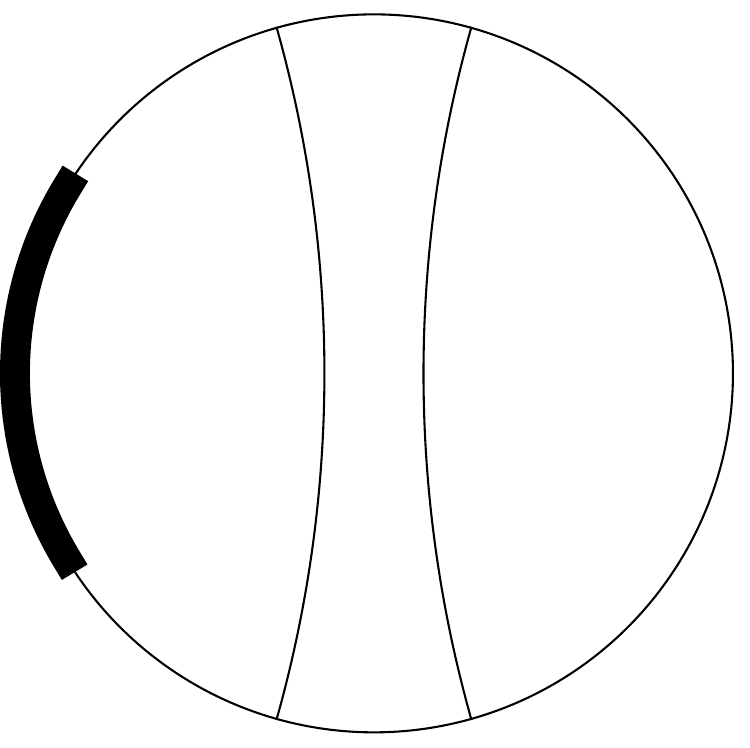}};
\draw[->,black, thick] (2,0) -- (3,0);
\node[] at (2.5,0.2) {$\sigma$};
\node (ex32) at (5,0){\includegraphics[width=1.5in]{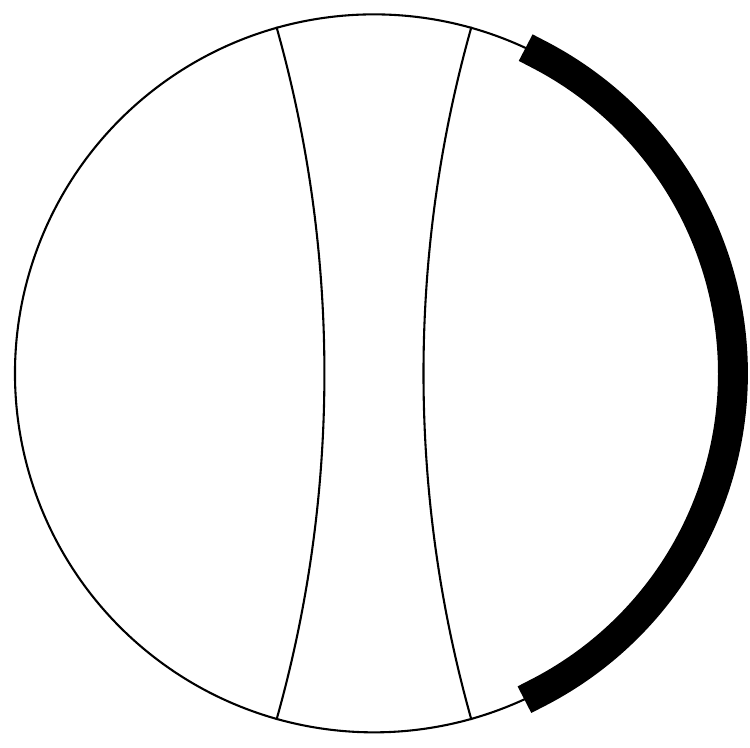}};
\draw[->,black, thick] (7,0) -- (8,0);
\node[] at (7.5,0.2) {$\sigma$};
\node (ex33) at (10,0){\includegraphics[width=1.5in]{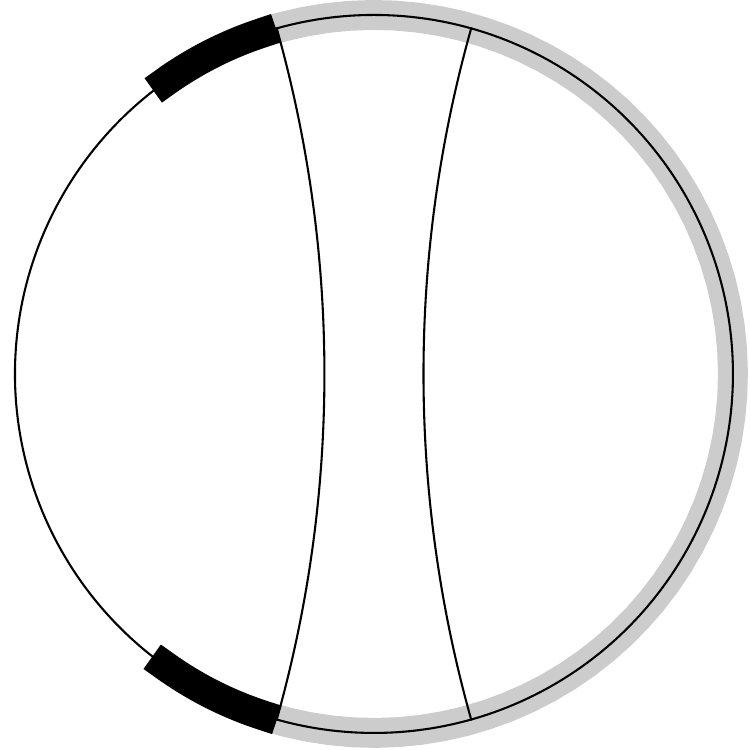}};

\draw[->,black,thick,rounded corners =5mm] (12,0) to (13,0) to (13,-3) to (-3,-3) to (-3,-6) to (-2,-6);
\node[] at (5,-2.7) {$\sigma$};

\node(ex34) at (0,-6){\includegraphics[width=1.5in]{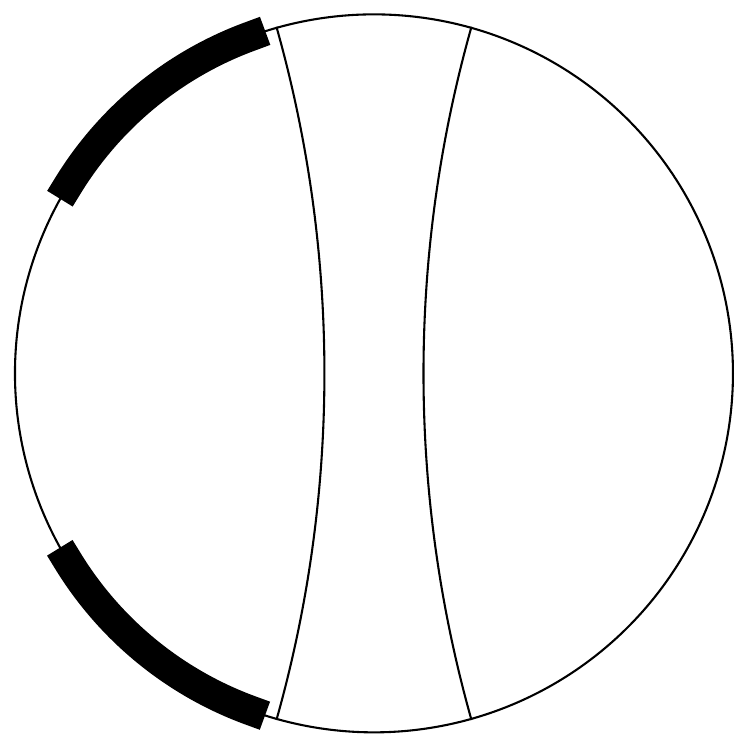}};
\draw[->,black, thick] (2,-6) -- (3,-6);
\node[] at (2.5,-5.7) {$\sigma$};
\node(ex35) at (5,-6){\includegraphics[width=1.5in]{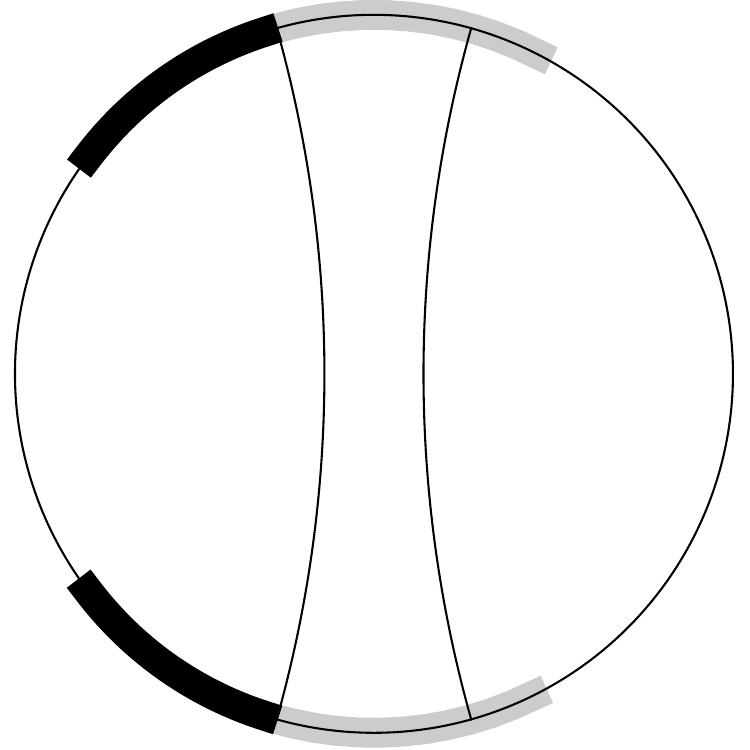}};
\draw[->,black, thick] (7,-6) -- (8,-6);
\node[] at (7.5,-5.7) {$\sigma$};
\node(ex35) at (10,-6){\includegraphics[width=1.5in]{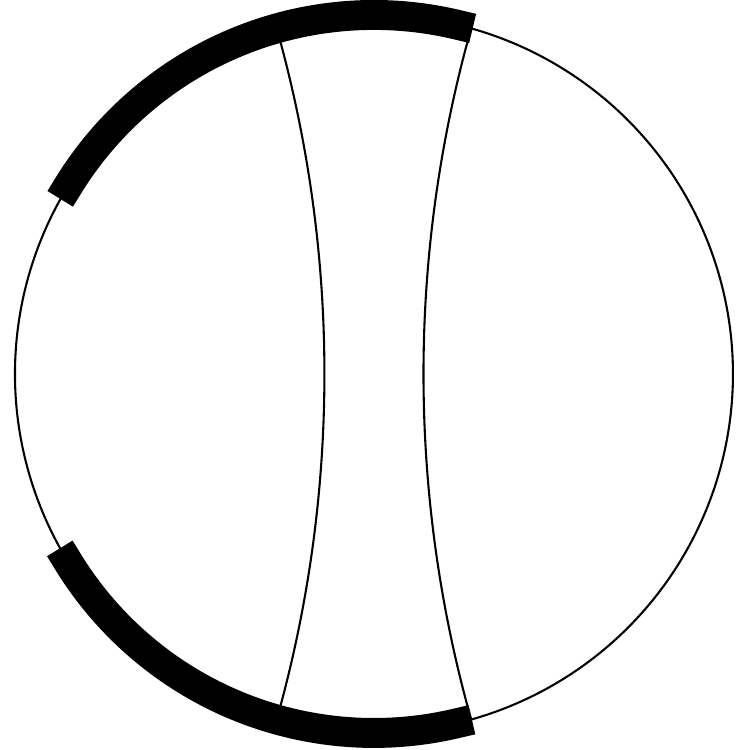}};
\node(52) at (-1.8,1.3) {52};
\node(74) at (-1.8,-1.3) {74};
\node(22) at (5.8,2) {22};
\node(104) at (5.8,-2) {104};
\node(37) at (9.5,2.1) {37};
\node(89) at (9.5,-2.1) {89};
\node(44) at (8.7,1.8) {44};
\node(82) at (8.7,-1.8) {82};
\node(52) at (-1.9,-5.) {52};
\node(74) at (-1.9,-7.) {74};
\node(38) at (-0.6,-3.9) {38};
\node(88) at (-0.6,-8.04) {88};

\node(37) at (4.4,-3.9) {37};
\node(89) at (4.4,-8.04) {89};
\node(50) at (3.14,-4.8) {50};
\node(76) at (3.14,-7.2) {76};
\node(22) at (6,-4) {22};
\node(108) at (6,-8.0) {104};

\node(26) at (10.6,-3.8) {26};
\node(44) at (8.1,-4.9) {52};
\node(82) at (8.1,-7.1) {74};
\node(100) at (10.6,-8.1) {100};

\end{tikzpicture}

\caption{$Q_i$ (black) and $R_i$ (grey and black) for $1 \le i \le 6$ for $H_6$.}
\label{fig:ex3}
\end{figure}

\medskip

\noindent
\textit{Example 4.} Consider the hyperbolic component $H_6'$ with $H'_4\triangleright H_6'$ and $H'_5\triangleright H_6'$, where 
\begin{enumerate}
\item $\Pi_1(H_6')=[10/63, 17/63]$, the $\star$-piece is the union $[10/126, 17/126]\cup [73/126, 80/126]$, the $A$-piece is $(80/126, 10/126)$, the $B$-piece is $(17/126, 73/126)$, $K(H_6')=BBABBB$ and $\mathrm{per}(H)=6$, 
\item $\Pi_1(H'_4)=[3/15, 4/15]$, $K(H'_4)=BBAA$ and $\mathrm{per}(H'_4)=4$,
\item $\Pi_1(H'_5)=[5/31, 6/31]$, $K(H'_5)=BBABA$ and $\mathrm{per}(H'_5)=5$.
\end{enumerate}
Below, we only write the numerators over the denominator $126$. Then,
\[\mathrm{int}(\Pi_1(H_6'))=(20^e, 34^e)\mapsto (40^e, 68^e)\mapsto (80^e, 10^e)\mapsto (34^e, 20^e)\]
whose coding is either $BBAA$, $BBA\,\star$ or $BBAB$. The $BBAB$-part of the last image is mapped as
\[(34^e, 73)\cup(17, 20^e)\mapsto (68^e, 20)\cup(34, 40^e)\]
whose coding is either $BBABA$, $BBAB\,\star$ or $BBABB$. The $BBABB$-part of the last image is mapped as
\[(68^e, 73)\cup(17, 20)\cup(34, 40^e)\mapsto (10^e, 20)\cup(34, 40)\cup(68, 80^e)\]
whose coding is either $BBABBB$ or $BBABB\,\star$. This is illustrated in Figure \ref{fig:ex4}.

\begin{figure}
\centering
\begin{tikzpicture}

\node(ex41) at (0,0){\includegraphics[width=1.5in]{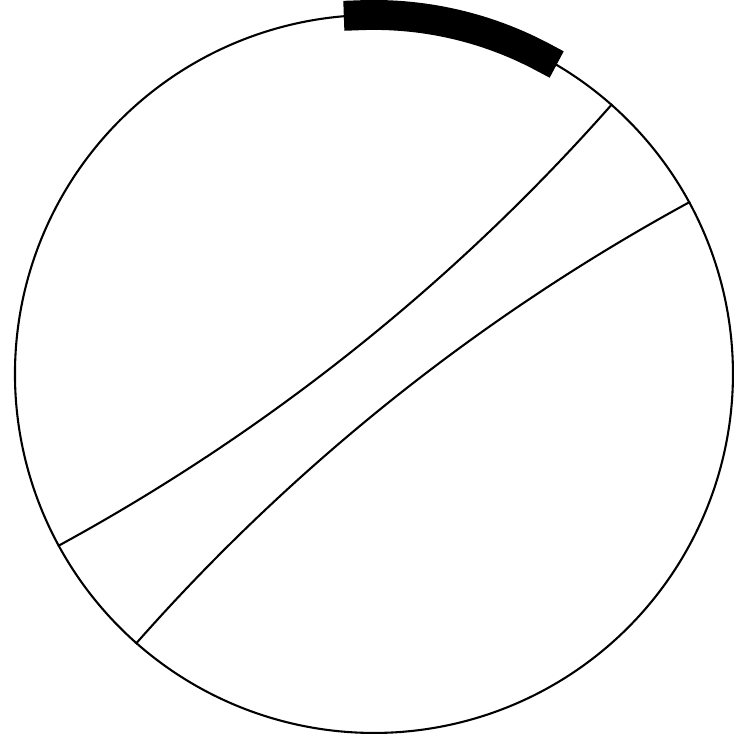}};
\draw[->,black, thick] (2,0) -- (3,0);
\node[] at (2.5,0.2) {$\sigma$};
\node (ex42) at (5,0){\includegraphics[width=1.5in]{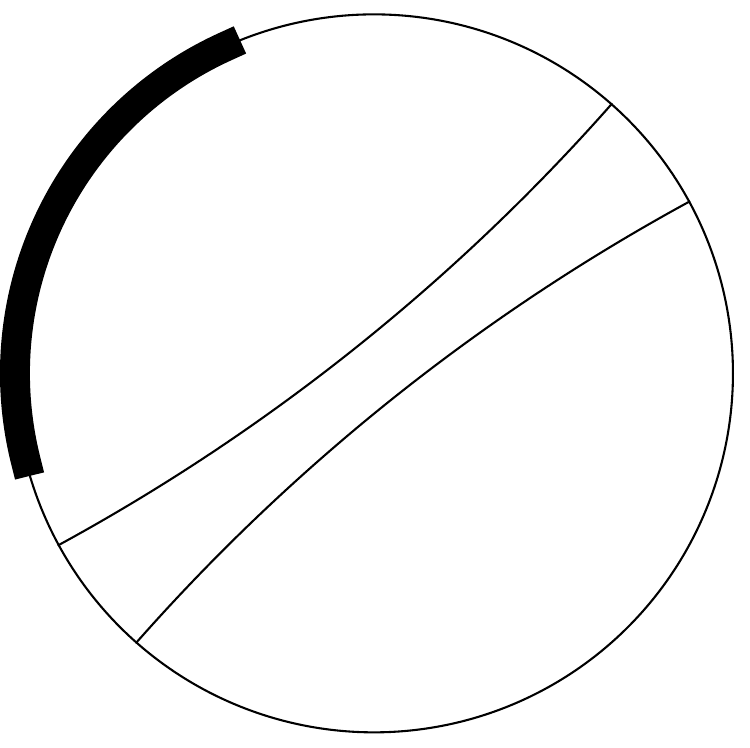}};
\draw[->,black, thick] (7,0) -- (8,0);
\node[] at (7.5,0.2) {$\sigma$};
\node (ex43) at (10,0){\includegraphics[width=1.5in]{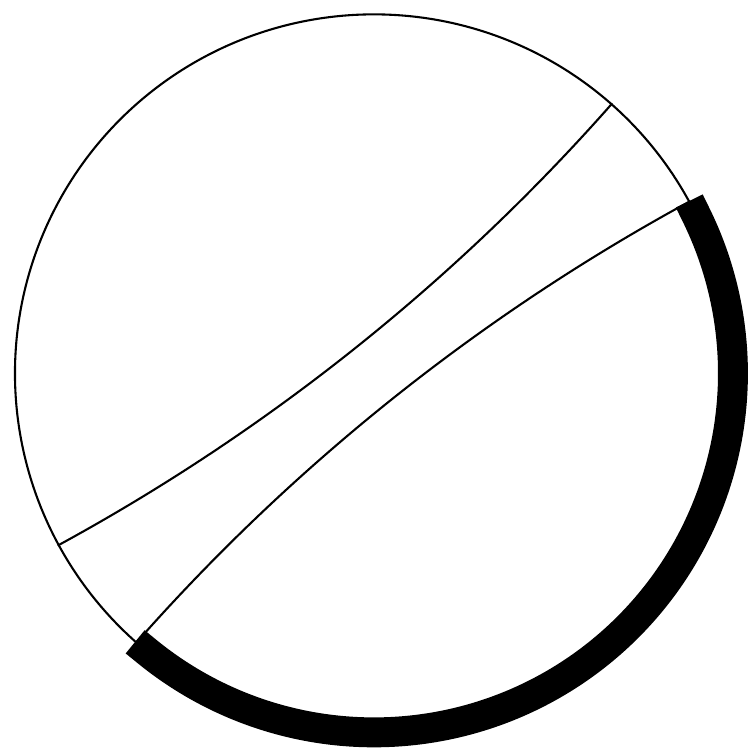}};

\draw[->,black,thick,rounded corners =5mm] (12,0) to (13,0) to (13,-3) to (-3,-3) to (-3,-6) to (-2,-6);
\node[] at (5,-2.7) {$\sigma$};

\node(ex44) at (0,-6){\includegraphics[width=1.5in]{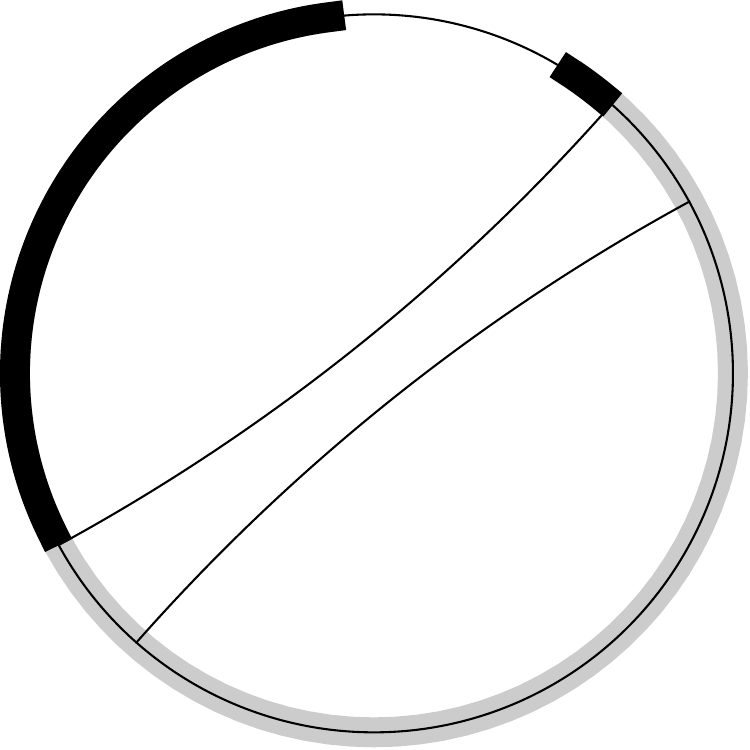}};
\draw[->,black, thick] (2,-6) -- (3,-6);
\node[] at (2.5,-5.7) {$\sigma$};
\node(ex45) at (5,-6){\includegraphics[width=1.5in]{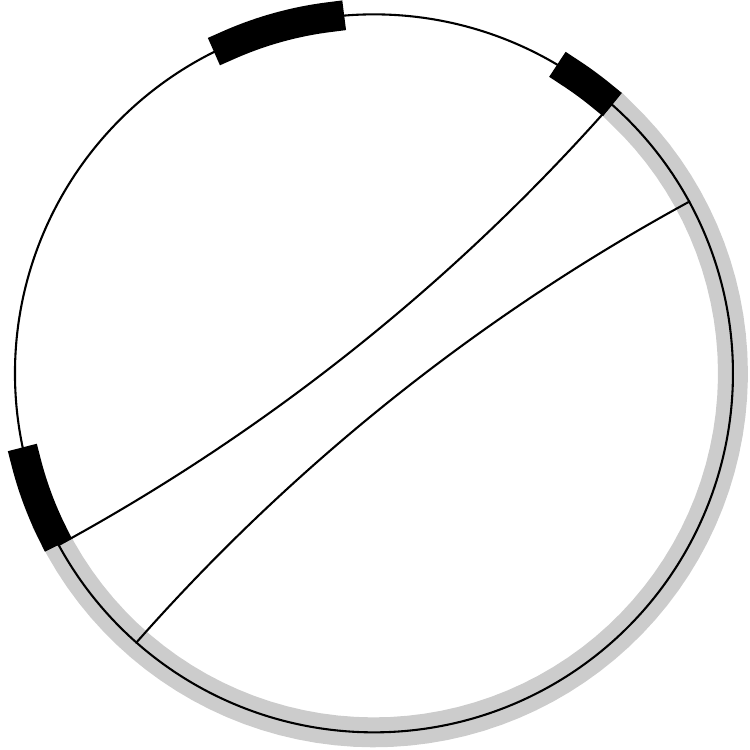}};
\draw[->,black, thick] (7,-6) -- (8,-6);
\node[] at (7.5,-5.7) {$\sigma$};
\node(ex45) at (10,-6){\includegraphics[width=1.5in]{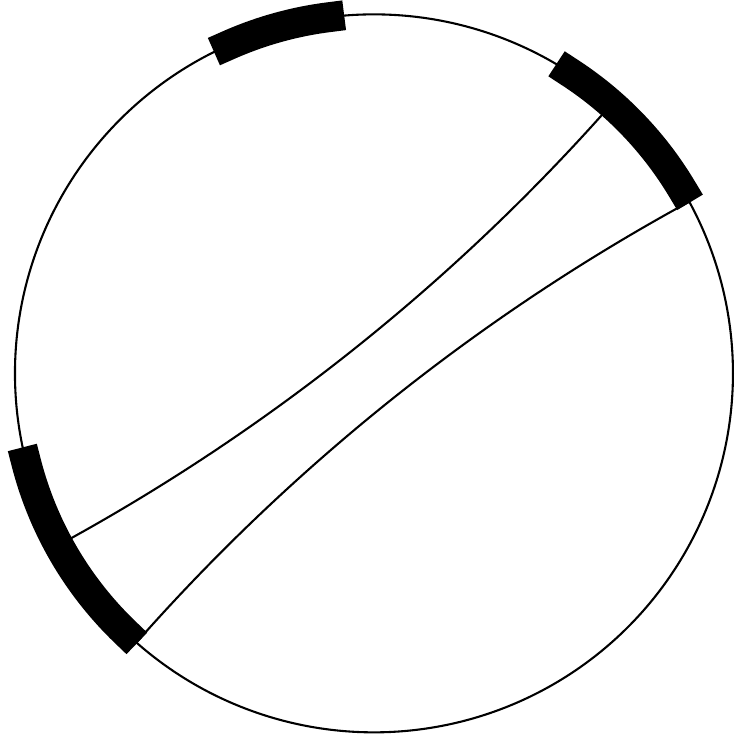}};

\node(20) at (1.3,1.8) {20};
\node(34) at (-0.3,2.2) {34};
\node(40) at (4.3,2.) {40};
\node(68) at (2.9,-0.7) {68};
\node(10) at (12,1.1) {10};
\node(80) at (8.5,-1.7) {80};
\node(20) at (1.,-4.1) {20};
\node(17) at (1.5,-4.5) {17};
\node(34) at (-0.3,-3.9) {34};
\node(73) at (-2.,-7) {73};
\node(68) at (2.9,-6.4) {68};
\node(73) at (3.,-7) {73};
\node(40) at (4.,-4.1) {40};
\node(34) at (4.8,-3.9) {34};
\node(20) at (6,-4.1) {20};
\node(17) at (6.5,-4.5) {17};

\node(10) at (11.9,-5.) {10};
\node(20) at (11,-4.1) {20};
\node(40) at (9.,-4.1) {40};
\node(34) at (9.8,-3.9) {34};
\node(68) at (7.9,-6.4) {68};
\node(80) at (8.5,-7.6) {80};
\end{tikzpicture}

\caption{$Q_i$ (black) and $R_i$ (grey and black)  for $1 \le i \le 6$ for $H_6'$.}
\label{fig:ex4}
\end{figure}
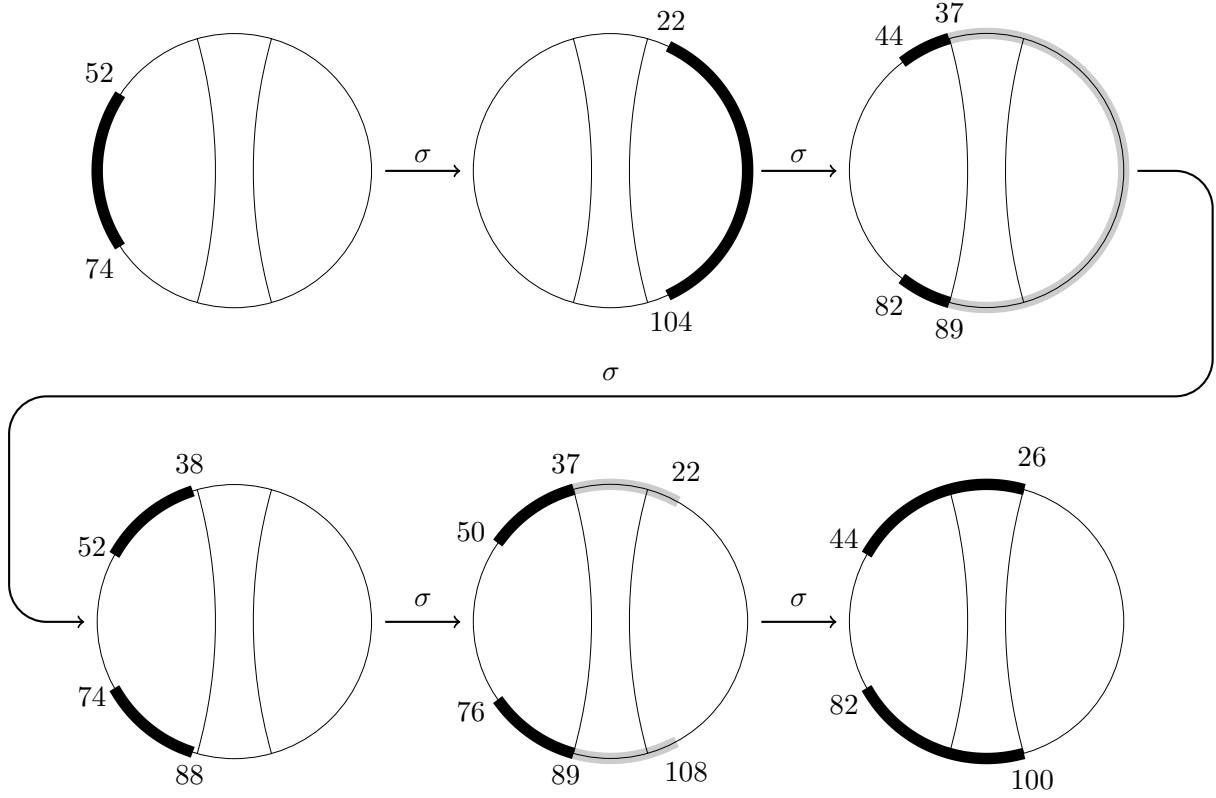

\section{The monodromy action conjecture}\label{sec:Henon}

In this Appendix we introduce the monodromy action conjecture of Lipa~\cite{Li} for the complex H\'enon family $f_{c, b}$ and explain how Corollary \ref{cor:main} gives rise to a solution to the degenerate case $b=0$ of the conjecture.

Consider the \textit{complex H\'enon family}:
\[f_{c,b} : (x, y)\longmapsto (x^2+c-by, x)\]
defined on $\mathbb{C}^2$, where $(c,b)\in\mathbb{C}^2$ is a parameter. Note that when $b=0$, the dynamics of $f_{c, 0}$ reduces to the quadratic map $p_c(z)=z^2+c$. 

\begin{figure}
	\includegraphics[scale=0.222]{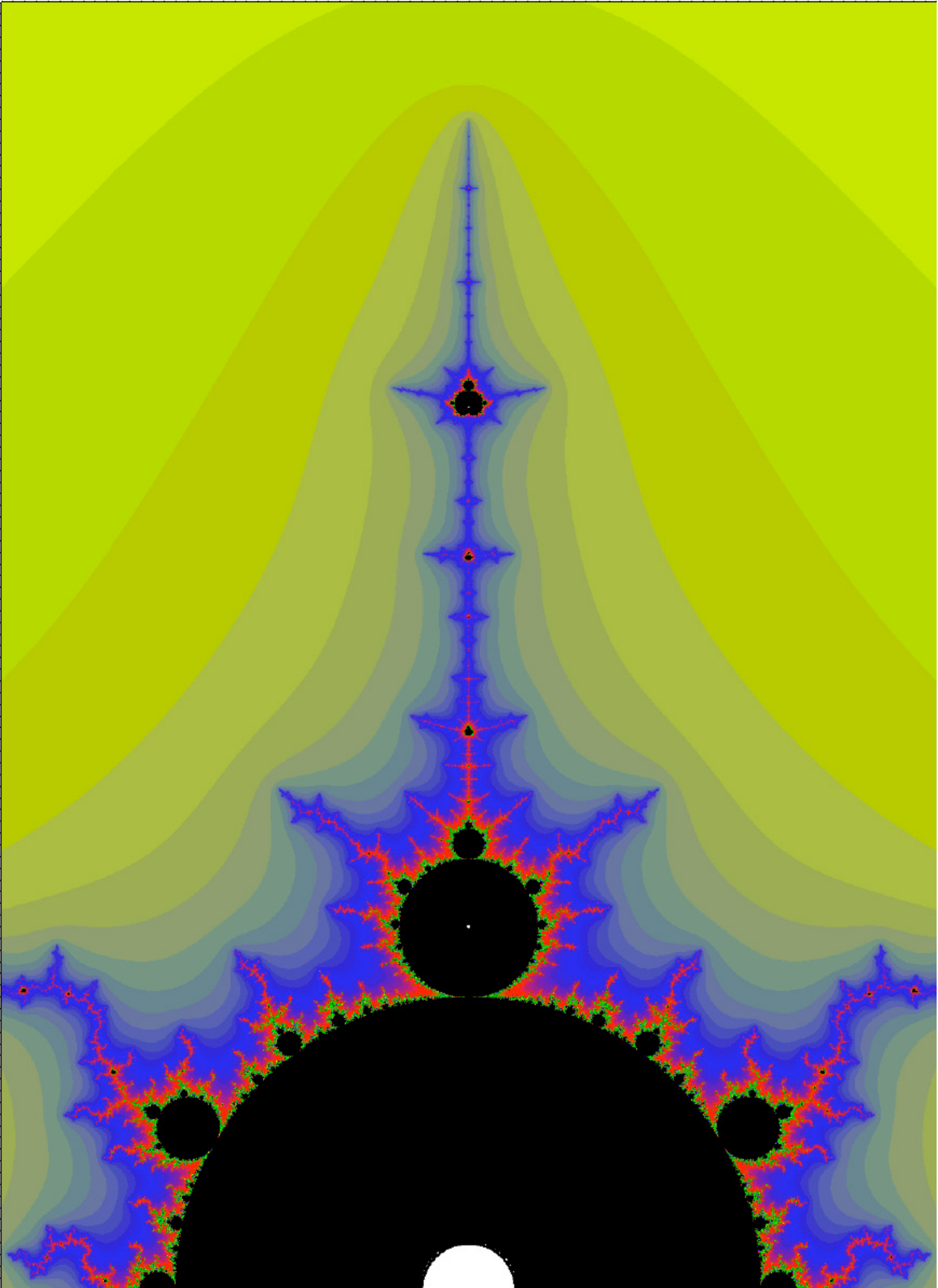}
	\includegraphics[scale=0.215]{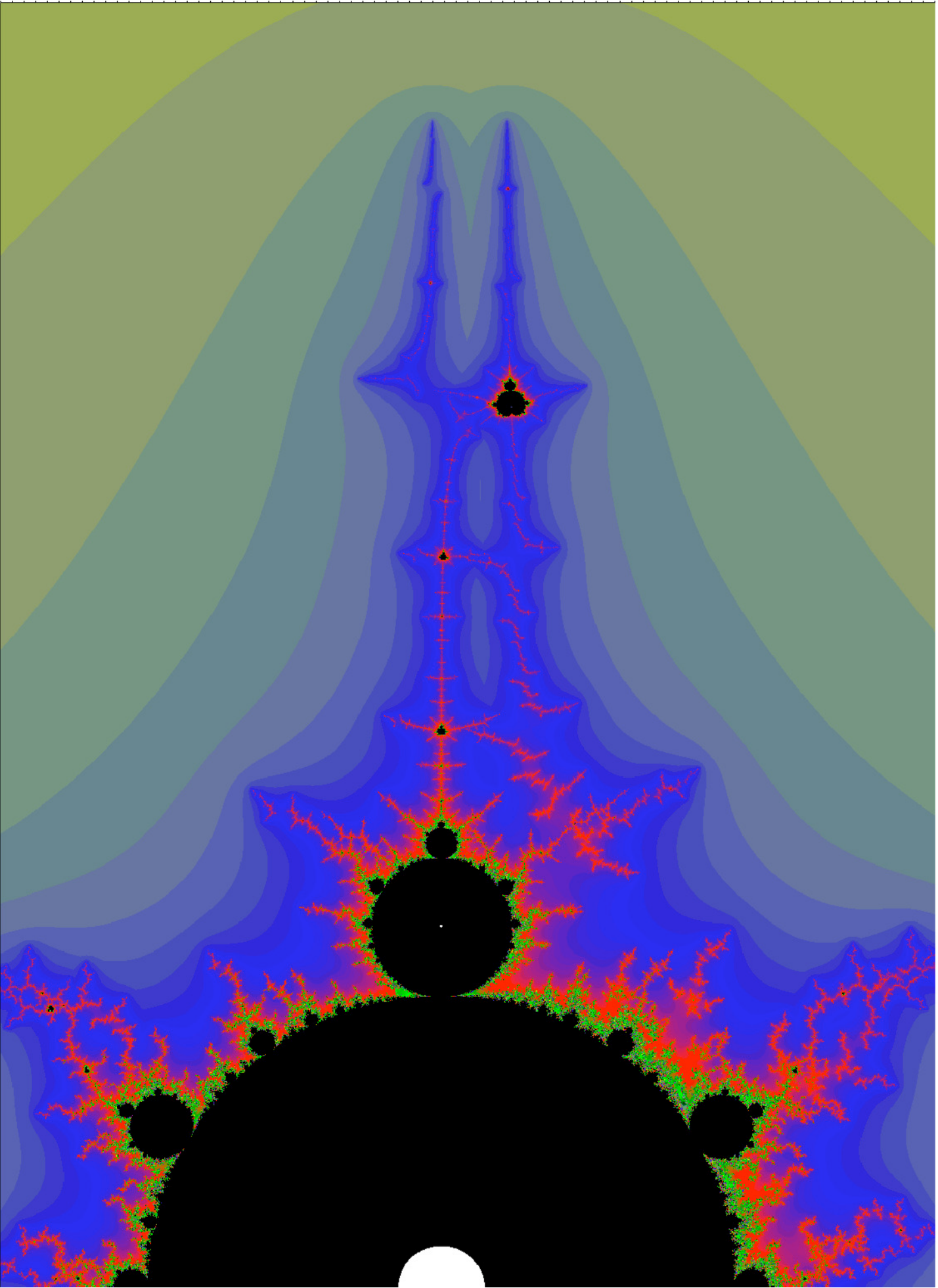}
	\includegraphics[scale=0.215]{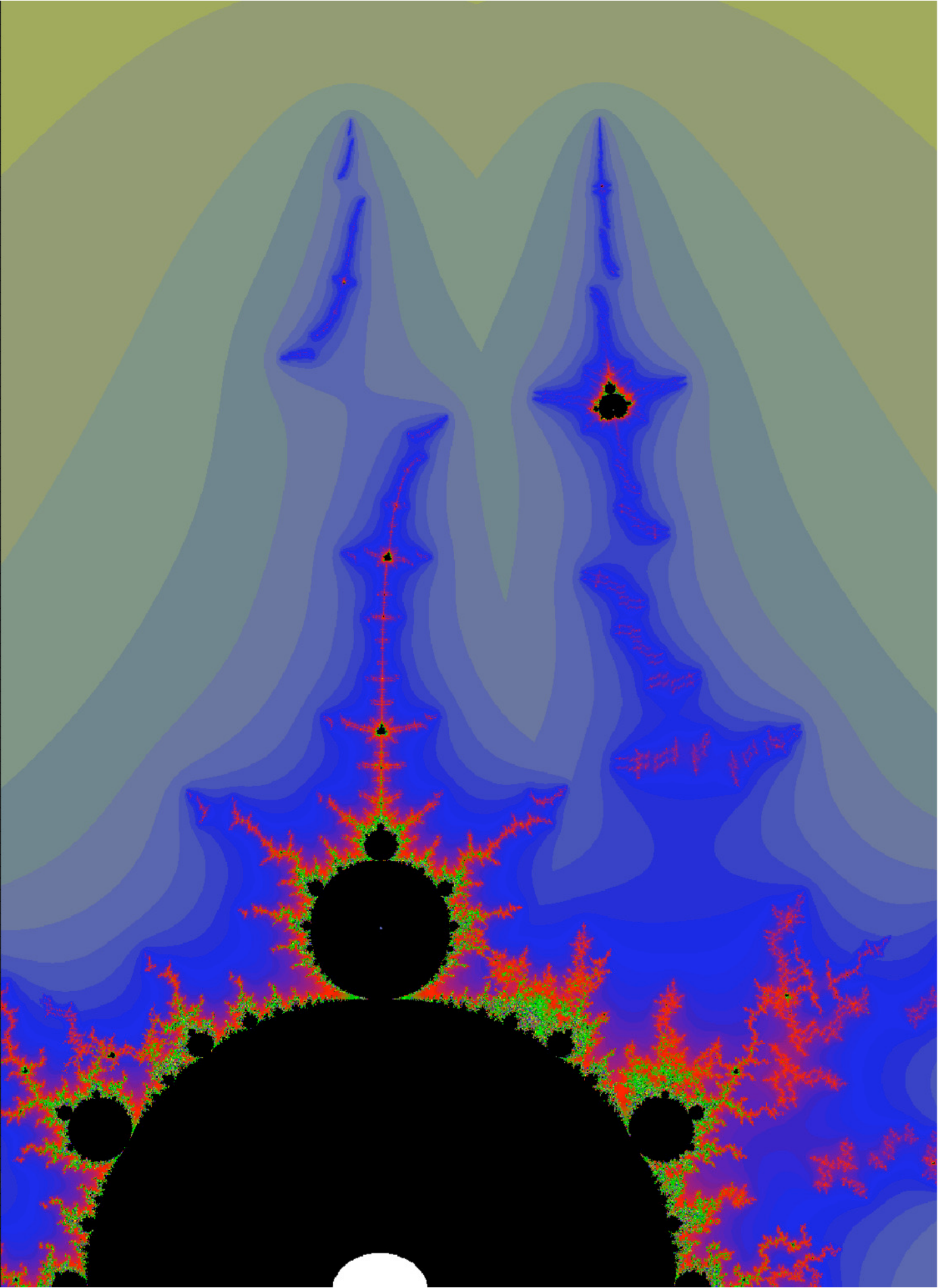}
	\caption{The $c$-planes with $b=0$ (left), $b=0.015i$ (middle) and $b=0.05i$ (right)~\cite{Li}.}
	\label{fig:splitting}
\end{figure}

We say that a complex H\'enon map is a \emph{hyperbolic horseshoe on} $\mathbb{C}^2$ if its Julia set $J_{c,b}$ is a hyperbolic set and $f_{c,b} : J_{c,b}\to J_{c,b}$ is topologically conjugate to the shift map $\sigma : \{A, B\}^{\mathbb{Z}}\to\{A, B\}^{\mathbb{Z}}$. Note that when $b=0$, we say that $f_{c, 0}$ is a hyperbolic horseshoe on $\mathbb{C}^2$ if its Julia set $J_c$ is a hyperbolic set and $p_c : J_c\to J_c$ is topologically conjugate to the shift map $\sigma : \{A, B\}^{\mathbb{N}}\to\{A, B\}^{\mathbb{N}}$. The \emph{complex hyperbolic horseshoe} locus is defined as
\[\mathcal{H}_{\mathbb{C}}\equiv\big\{(c,b)\in \mathbb{C}\times\mathbb{C}^{\times} : f_{c,b} \mbox{ is a hyperbolic horseshoe on } \mathbb{C}^2 \bigr\}.\]
%
Since we do not know if $\mathcal{H}_{\mathbb{C}}$ is connected, we define the \emph{shift locus} $\mathcal{S}_{\mathbb{C}}$ for the complex H\'enon family as the connected component of $\mathcal{H}_{\mathbb{C}}$ containing $(\mathbb{C}\setminus\mathfrak{M})\times \{0\}$.
In this region, Lipa aimed to understand the monodromy homomorphism:
\[\rho : \pi_1(\mathcal{S}_{\mathbb{C}}, \ast) \longrightarrow \mathrm{Aut}(\sigma, \{A, B\}^{\mathbb{Z}})\]
associated to a specific class of loops. His work concerns loops wrapping around the empirically observed phenomenon of ``herds.''

The software \emph{SaddleDrop} produced by Papadantonakis and Hubbard allows a user to create approximate pictures of 2D slices of the H\'enon parameter space~\cite{K}. Using this, Koch  observed (personal communication, 2012) that, as we perturb our Jacobian parameter $b$ away from zero, the Mandelbrot set  in the slice $\{(c,0) : c\in\mathbb{C}\}$ breaks up and renormalized Mandelbrot sets break off, moving in a direction related to the direction of parameter perturbation. These renormalized Mandelbrot sets move together in \emph{herds}. These have some coding based on levels of splitting, there is some hierarchy that is observed.  

Consider, first, the slice $\{(c,0) : c\in\mathbb{C}\}$ in parameter space (see the left picture in Figure~\ref{fig:splitting}) we see the Mandelbrot set $\mathfrak{M}$ in the $c$-plane. If we change $b$ slightly, we still have a Mandelbrot-like set in the corresponding $c$-plane. As we perturb $b$ more, the components beyond the root point of the Airplane hyperbolic component, $H_{\mathrm{Air}}$, seems to split into two different pieces (see the middle picture in Figure~\ref{fig:splitting}.) One, called the \emph{$A$-herd} of $H_{\mathrm{Air}}$, containing all the hyperbolic components beyond $H_{\mathrm{Air}}$ whose kneading sequences end in $A$, moves in the $c$-plane in the direction in which $b$ is perturbed. The other, called the \emph{$B$-herd} of $H_{\mathrm{Air}}$, containing all the hyperbolic components beyond $H_{\mathrm{Air}}$ whose kneading sequence end in $B$, moves in the opposite direction (see the right picture in Figure~\ref{fig:splitting}.) This speculative structure is discussed further in Section 8.1 of~\cite{Li}.

\begin{figure}
	\includegraphics[scale=0.296]{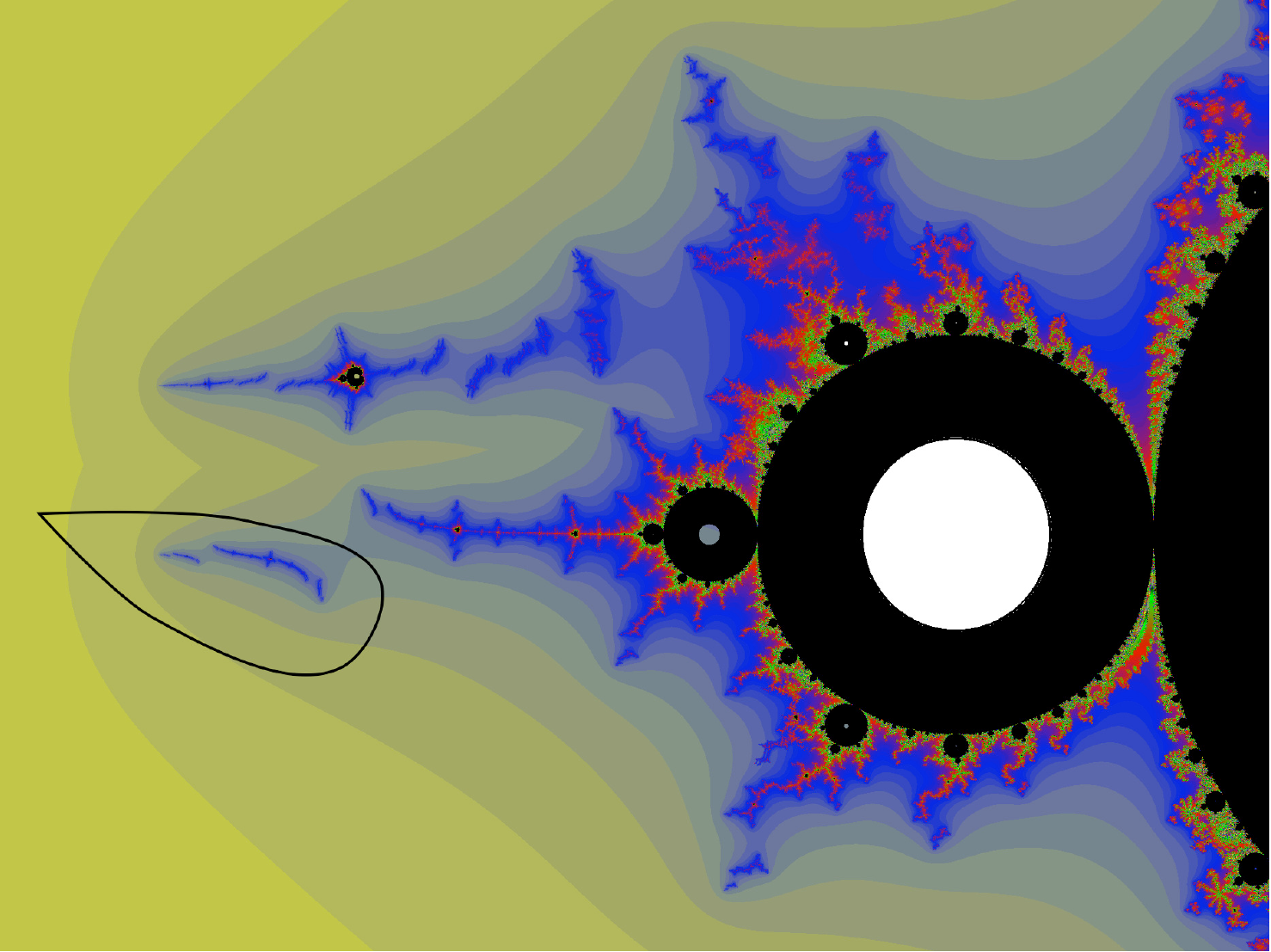}
	\includegraphics[scale=0.315]{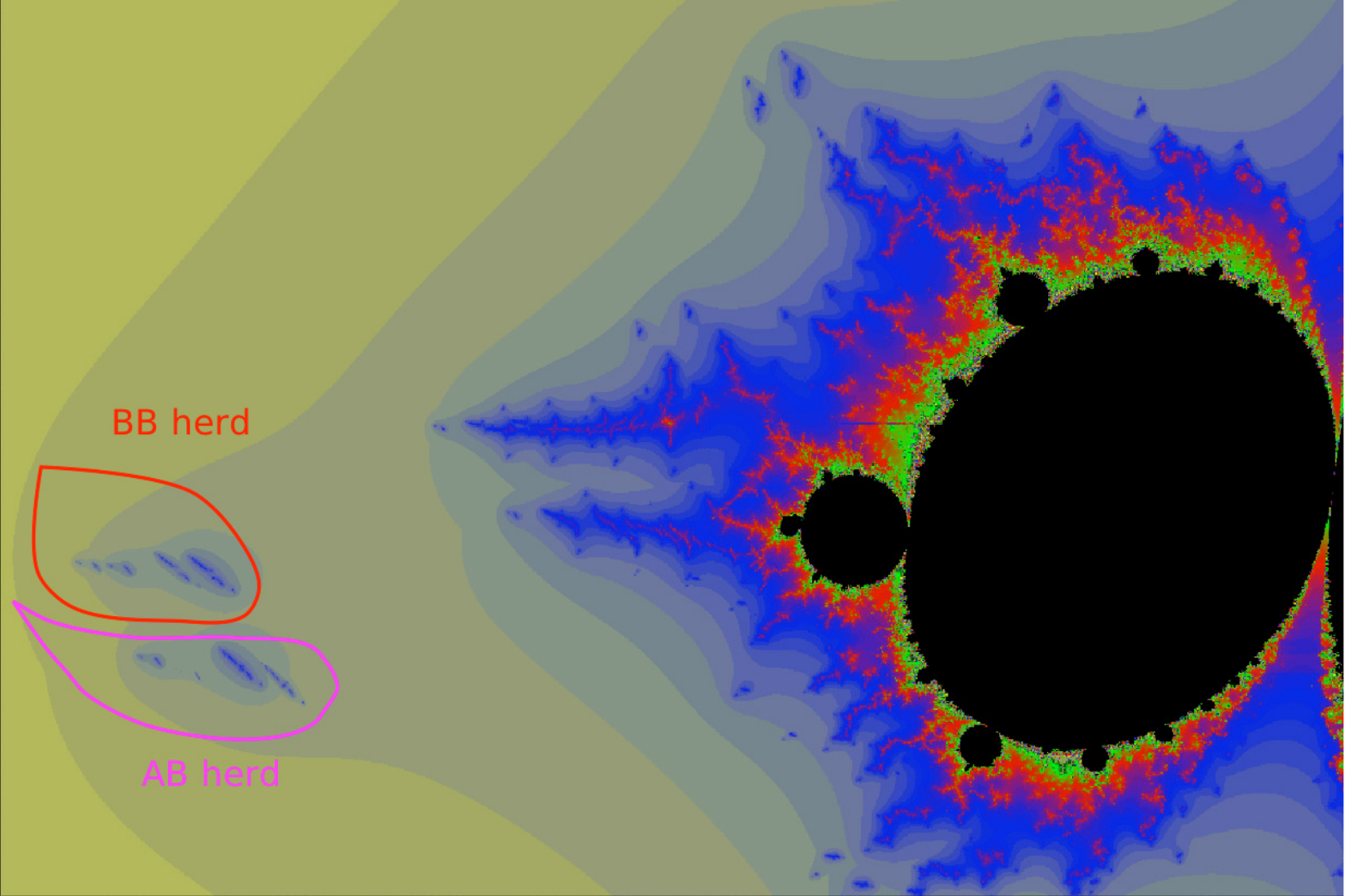}
	\caption{A loop surrounding the $B$-herd of Airplane wake with $b=0.05i$ (left) which splits to the $BB$-herd and the $AB$-herd with $b=0.2+0.3i$ (right)~\cite{Li}.}
	\label{fig:herds}
\end{figure}

This coding of herds forms part of the monodromy action conjectured automorphism associated with some loop in parameter space. The other part of the coding comes from the kneading sequence of the wake from which a herd originates. In Figure~\ref{fig:herds}, we see examples of loops around herds. Lipa conjectures that the marker describing the monodromy action of these loops is given by the coding of the herd followed by a $\ast$ entry, followed by the kneading sequence of the one-dimensional hyperbolic component. In the examples in Figure~\ref{fig:herds}, this would mean the markers are $B\ast BAA$, $BB \ast BAA$, and $AB\ast BAA$. More generally, he conjectures the following:

\begin{cjt}
Let $H$ be a hyperbolic component of $\mathfrak{M}$ and let $\underline{v}$ be a word over $\{A, B\}$. Suppose that $\gamma\in\pi_1(\mathcal{H}, \ast)$ winds around the $\underline{v}$-herd of the wake $\mathcal{W}_H$. Let $H_1, \dots, H_L$ be the hyperbolic components conspicuous to $H$ and put $\underline{w}^i\equiv \underline{v}\ast K(H_i)$. Then, the monodromy action of $\rho(\gamma)$ is described by the compositions of the markers $W\equiv\{\underline{w}^1, \dots, \underline{w}^L\}$.
\label{cjt:lipa}
\end{cjt}


In the degenerate case $b=0$ we think of our loops as wrapping around all possible herds and so the left-hand side of the marker, arising from the herd coding, would be every possible word on $\{A,B\}$. This is equivalent to dropping the left-hand word. In this case, the monodromy action conjecture would say that the pseudo-monodromy associated to some pseudo-loop passing through the root point of a hyperbolic component $H$ would be described by the collection of markers $\ast \,  K(H_i)$ for $H_i \triangleright H$, and compositions of these. Thus the markers we obtain from Corollary \ref{cor:main} agree partly with the markers conjectured by Lipa. One difference is that we obtain a marker of infinite length, which is not present in Lipa's formulation, and also markers of the form $\widehat{K}(H_i)\, \star$. In Chapter 5 of his thesis, Lipa considers the degenerate case, but only looks at the points in the Julia set that never re-enter $\Pi_0(H)$. In the two-dimensional case, he considers a continuous extension of this set, thus he is only considering a subset of the Julia set. Restricting to such a subset, Corollory \ref{cor:main} agrees with, and proves the monodromy action conjecture in the degenerate case. However, the additional markers we obtain are necessary to describe the monodromy of the entire Julia set.  

We expect to be able to extend our result to the H\'enon setting, at least for small perturbations, taking the inverse limit of our one-dimensional markers to describe the monodromy associated to specific loops in a $\{b=\varepsilon\}$-slice of the parameter space for $\varepsilon\in\mathbb{C}$ close to $0$. More specifically, if a gap appears allowing us to take a loop around all herds of a hyperbolic component, then the monodromy should be described by the inverse limit of markers stated above. This is supported by evidence from Lipa's thesis, in particular, Theorem 6.22. This suggests that infinite markers, not predicted by Lipa, will be necessary to describe the monodromy.


In full generality, the the monodromy action conjecture is far from being resolved, in part because the foundation of the conjecture relies on empirically observed  phenomenon of herds, and there is no concrete way to obtain the coding of a herd without following the perturbation from the Mandelbrot set. The recent PhD thesis of Richards~\cite{R} improved this point, and performed more precise computation on the statement of the the monodromy action conjecture. Combining this with our understanding of the one-dimensional case may lead to a proof of the monodromy action conjecture at least for H\'enon maps which can be modelled by building up finitely many quadratic polynomials. This will be the subject of our future work.

\medskip

\noindent

\textbf{Acknowledgements.}
Both authors of this article thank Mitsuhiro Shishikura and John Smillie for fruitful advice and comments on this manuscript. They are also grateful to the anonymous referees for their comments which improved the manuscript. YI is partially supported by JSPS KAKENHI Grant Numbers 20H01809 and 22H05107, and TR is partially supported by JSPS KAKENHI Grant Number 22H05107.

\end{document}